\documentclass{amsart}
\usepackage[utf8]{inputenc}
\usepackage{amsmath}
\usepackage{amssymb}
\usepackage{mathrsfs}
\usepackage{graphicx}

\usepackage[colorlinks]{hyperref}
\usepackage[svgnames, dvipsnames, usenames]{xcolor}
\hypersetup{urlcolor = RedViolet, linkcolor = RoyalBlue, citecolor = ForestGreen}

\usepackage{tikz}
\usetikzlibrary{intersections}

\usepackage{partmac}
\usepackage{graphmor}

\def\ppen{\penalty 300 }
\let\col=\colon
\def\colon{\col\ppen}

\theoremstyle{plain} 
\newtheorem{res}{Result}
\newtheorem{thm}{Theorem}[section]
\newtheorem{prop}[thm]{Proposition}
\newtheorem{lem}[thm]{Lemma}
\newtheorem{cor}[thm]{Corollary}

\theoremstyle{definition}
\newtheorem{defn}[thm]{Definition}
\newtheorem{rem}[thm]{Remark}
\newtheorem{ex}[thm]{Example}
\newtheorem{quest}[thm]{Question}

\numberwithin{equation}{section}

\renewcommand{\theta}{\vartheta}
\renewcommand{\phi}{\varphi}
\renewcommand{\epsilon}{\varepsilon}
\renewcommand{\subset}{\subseteq}
\renewcommand{\supset}{\supseteq}

\newcommand{\N}{\mathbb N}
\newcommand{\Z}{\mathbb Z}

\newcommand{\R}{\mathbb R}
\newcommand{\C}{\mathbb C}

\newcommand{\F}{\mathcal F}
\newcommand{\staralg}{\mathop{\rm\ast\mathchar `\-alg}}
\DeclareMathOperator{\Mor}{Mor}
\DeclareMathOperator{\id}{id}
\DeclareMathOperator{\Aut}{Aut}
\DeclareMathOperator{\spanlin}{span}

\DeclareMathOperator{\Tr}{Tr}

\newcommand{\Cat}{\mathscr{C}}
\newcommand{\RCat}{\mathfrak{C}}
\newcommand{\Lin}{\mathscr{L}}
\newcommand{\Pair}{\mathsf{Pair}}
\newcommand{\NCPair}{\mathsf{NCPair}}
\newcommand{\Part}{\mathsf{Part}}
\newcommand{\NCPart}{\mathsf{NCPart}}
\newcommand{\Bipart}{\mathsf{Bipart}}
\newcommand{\NCBipart}{\mathsf{NCBipart}}
\newcommand{\NCBipartEven}{\mathsf{NCBipartEven}}
\newcommand{\Had}{\mathsf{Had}}
\newcommand{\NCHad}{\mathsf{NCHad}}
\newcommand{\Mat}{\mathsf{Mat}}
\newcommand{\Hilb}{\mathsf{Hilb}}
\newcommand{\Neg}{\mathscr{N}}
\newcommand{\GHadamard}{\Graph{\draw (1,0) -- (1,0.5) \Gmapsc{} -- (1,1);}}
\newcommand{\GGr}{\mathfrak{G}}
\newcommand{\HGr}{\mathfrak{H}}
\newcommand{\Olg}{\mathscr{O}}
\newcommand{\Alg}{\mathscr{A}}
\newcommand{\CTr}{\mathop{\mathsf{Tr}}}
\newcommand{\Ctrans}{^\mathsf{T}}
\newcommand{\tiltimes}{\mathbin{\tilde\times}}

\newcommand{\spider}{\Graph{\draw (3,0) -- (3,0.5) node{} -- (3,1);\draw (1,0) -- (1,0.2) -- (3,0.5); \draw (2,0) -- (2,0.2) -- (3,0.5); \draw (5,0) -- (5,0.2) -- (3,0.5); \draw (1,1) -- (1,0.8) -- (3,0.5); \draw (2,1) -- (2,0.8) -- (3,0.5); \draw (5,1) -- (5,0.8) -- (3,0.5); \node[draw=none,fill=none] at (4,0.1) {$\dots$}; \node[draw=none,fill=none] at (4,0.9) {$\dots$};}}

\newcommand{\TLid}{
\Graph{
\draw (0.9,0) -- (0.9,1);
\draw (1.1,0) -- (1.1,1);
}
}

\newcommand{\TLfork}{
\Graph{
\draw (0.9,0) .. controls (0.9,0.5) and (1.4,0.5) .. (1.4,1);
\draw (2.1,0) .. controls (2.1,0.5) and (1.6,0.5) .. (1.6,1);
\draw (1.1,0) .. controls (1.1,0.5) and (1.9,0.5) .. (1.9,0);
}
}

\newcommand{\LTfork}{
\Graph{
\draw (1.1,0) .. controls (1.1,0.5) and (1.6,0.5) .. (1.6,1);
\draw (1.9,0) .. controls (1.9,0.5) and (1.4,0.5) .. (1.4,1);
\draw (0.9,0) .. controls (0.9,0.5) and (2.1,0.5) .. (2.1,0);
}
}

\newcommand{\TLconnecter}{
\Graph{
\draw (0.9,0) -- (0.9,1);
\draw (2.1,0) -- (2.1,1);
\draw (1.1,0) .. controls (1.1,0.5) and (1.9,0.5) .. (1.9,0);
\draw (1.1,1) .. controls (1.1,0.5) and (1.9,0.5) .. (1.9,1);
}
}

\newcommand{\LTconnecter}{
\Graph{
\draw (1.1,0) -- (1.1,1);
\draw (1.9,0) -- (1.9,1);
\draw (0.9,0) .. controls (0.9,0.5) and (2.1,0.5) .. (2.1,0);
\draw (0.9,1) .. controls (0.9,0.5) and (2.1,0.5) .. (2.1,1);
}
}

\newcommand{\TLpair}{
\Graph{
\draw (0.9,0) .. controls (0.9,0.6) and (2.1,0.6) .. (2.1,0);
\draw (1.1,0) .. controls (1.1,0.4) and (1.9,0.4) .. (1.9,0);
} 
}

\begin{document}
\title{Quantum symmetries of Hadamard matrices}
\author{Daniel Gromada}
\address{Czech Technical University in Prague, Faculty of Electrical Engineering, Department of Mathematics, Technická 2, 166 27 Praha 6, Czechia}
\email{gromadan@fel.cvut.cz}
\thanks{I would like to thank István Heckenberger for drawing my attention to the book \cite{CK17} (and its authors for writing it). I thank to Simon Schmidt for pointing out a mistake in an earlier version of the manuscript as well as for discussions about graph quantum isomorphisms. I also thank to the referee of this article for careful reading and pointing out several things to improve and correct. Finally, I thank to Moritz Weber for inspiring discussions about Hadamard matrices and their symmetries.}
\thanks{This work was supported by the project OPVVV CAAS CZ.02.1.01/0.0/0.0/16\_019/0000778}
\date{\today}

\begin{abstract}
We define quantum automorphisms and isomorphisms of Hadamard matrices. We show that every Hadamard matrix of size $N\ge 4$ has quantum symmetries and that all Hadamard matrices of a fixed size are mutually quantum isomorphic. These results pass also to the corresponding Hadamard graphs. We also define quantum Hadamard matrices acting on quantum spaces and bring an example thereof over matrix algebras.
\end{abstract}

\maketitle
\section*{Introduction}
The original motivation for this work and the main tool used here is a certain diagrammatic category or diagrammatic calculus developed recently independently in several different contexts. In this article, the category is denoted by $\Bipart_N=\langle\Gfork,\Gwhite\Gfork,\crosspart,\pairpart\rangle$ and the generators $\Gfork$ and $\Gwhite\Gfork$ are called \emph{complementary spiders}. The diagrammatic calculus for complementary spiders was first developed by Coecke and Duncan in the area of categorical quantum mechanics \cite{CD07,CK17}. The main object of our focus is, however, the subcategory $\NCBipartEven_N=\langle\Gconnecter,\Gwhite\Gconnecter,\pairpart\rangle\subset\Bipart_N$, which was recently introduced in \cite{GD4} in an attempt to find a liberated quantum group analogue of Coxeter groups of type $D$. In \cite{GD4}, a supposedly new quantum group $D_4^+$ is defined, which is a non-classical liberation of $D_4$ (the Coxeter group of type $D$, not the dihedral group). Although this group looked new and somewhat free, we are actually going to show here that it is isomorphic to $SO_4^{-1}$ (the anticommutative deformation of $SO_4$). The representation category of $SO_4$ (which is isomorphic, but not monoidally isomorphic to the representation category of $SO_4^{-1}$) was also recently independently constructed in \cite{CE23}.

In this article we view diagrammatic categories from the perspective of quantum groups. Quantum groups form the analogue of groups in non-commutative geometry. According to the Woronowicz--Tannaka--Krein theorem \cite{Wor88}, representations of quantum groups form a certain monoidal $\dag$-category and, conversely, any such category gives rise to a quantum group. Our goal for this article is to interpret the diagrammatic category $\NCBipartEven_N$ in terms of quantum groups. More precisely, we are going to interpret the quantum groups associated to $\NCBipartEven_N$ as quantum symmetries of certain classical objects.

It turns out that these objects are Hadamard matrices. Moreover, note that there is a concept of Hadamard graphs, which are constructed in such a way that their symmetries correspond to symmetries of Hadamard matrices. It turns out that this also holds for quantum symmetries, so we can model these using $\NCBipartEven_N$ as well. The following summarizes the results of Sections \ref{secc.hadamard}, \ref{secc.graphs}.

\begin{res}\label{res.1}
Let $H$ be a Hadamard matrix of size $N$. The representation category of its quantum automorphism group can be modelled by $\NCBipartEven_N=\langle\Gconnecter,\Gwhite\Gconnecter,\pairpart\rangle$. We can say the same about the quantum automorphism group of the associated looped Hadamard graph. For $N\ge 4$, it is a proper quantum group, not a group, so Hadamard matrices and Hadamard graphs have genuine quantum symmetries.
\end{res}

As was pointed out to the author by Simon Schmidt and also the referee of this article, there are Hadamard matrices, whose classical automorphism group is\footnote{See http://neilsloane.com/hadamard/. For instance, matrices 28.7, 28.8, 28.9 have automorphism group $\Z_2$.} $\Z_2$. Our result shows that they must have quantum symmetries. There is a famous open problem in the field of quantum automorphisms of graphs whether there is a graph with trivial automorphism group, but non-trivial quantum automorphism group. This example brings us very close by solving \cite[Problem~3.10]{Web23}.

The word \emph{modelled} used in Result~\ref{res.1} means that there is an appropriate fibre functor $F_H$ for every Hadamard matrix $H$, which surjectively maps $\NCBipartEven_N$ to the representation category of the quantum symmetry group. Moreover, the diagrammatic category has enough ``reduction rules'' so that every closed diagram can be reduced to a number. As a result, every such fibre functor must be also injective up to negligible morphisms. This has a remarkable consequence: For fixed $N$, the representation categories of quantum automorphism groups of all Hadamard matrices are monoidally equivalent. In the language of quantum isomorphisms:

\begin{res}[Theorems \ref{T.hadamard}, \ref{T.graphs}]\leavevmode
\begin{enumerate}
\item All Hadamard matrices of a fixed size are mutually quantum isomorphic.
\item All Hadamard graphs of a fixed size are mutually quantum isomorphic.
\end{enumerate}
\end{res}

Shortly before this paper was finished, the result on quantum isomorphism of Hadamard graphs was independently obtained by Chan and Martin \cite{CM22}. In their work, it appears as a consequence of a more general result on association schemes.

The notion of \emph{quantum isomorphism} of graphs actually first came from quantum information theory \cite{AMR19} and it is currently quite a popular topic in both quantum information and quantum groups. Some examples of pairs of graphs, which are not isomorphic, but are quantum isomorphic are known already \cite{AMR19,RS22,Sch22,MRV19}. Nevertheless, as far as we know, this is the first known example of more than two graphs being mutually quantum isomorphic.

Motivated by the results above, we also define the quantum version of Hadamard matrices and Hadamard graphs. If $X$ is a finite quantum space (that is, a special Frobenius $*$-algebra or, equivalently, a finite-dimensional C*-algebra equipped with a certain state), then a quantum Hadamard matrix is a linear map $H\colon l^2(X)\to l^2(X)$ satisfying certain properties. We also bring an example of such a Hadamard matrix:

\begin{res}[Example~\ref{E.transpose}]
Consider the finite quantum space $M_n$ given by the algebra of all $n\times n$ matrices. Then the transposition $a\mapsto a^{\rm T}$ taken as a linear map $l^2(M_n)\to l^2(M_n)$ is a normalized quantum Hadamard matrix.
\end{res}

The motivation for introducing quantum Hadamard matrices is that our results for classical Hadamard matrices hold for the quantum ones as well. Their quantum automorphism group is again given by $\langle\Gconnecter,\Gwhite\Gconnecter,\pairpart\rangle$, we can again define the corresponding quantum Hadamard graphs and two Hadamard matrices/graphs are quantum isomorphic if and only if the underlying finite quantum spaces are quantum isomorphic.

We also studied the category $\NCBipartEven_N$ on its own and we were able to determine its structure as a certain cartesian product of Temperley--Lieb categories. As a consequence, we have that the fibre functor on $\NCBipartEven_N$ is not only injective up to negligible morphisms, but truly injective. This is because there are actually no negligible morphisms (if $N\ge 4$).

\begin{res}[Theorem~\ref{T.Catalan}, Corollary~\ref{C.NoNegligible}]
The category $\NCBipartEven_N$ is isomorphic (but not monoidally isomorphic) to $\NCPair_{\sqrt N}\times\NCPair_{\sqrt N}$, where $\NCPair$ denotes the Tem\-perley--Lieb category of all non-crossing pairings. Consequently, any fibre functor on $\NCBipartEven_N$ is injective for $N\ge 4$.
\end{res}

%As a final remark, we would like to point the reader's attention to how this work connects quantum information theory with quantum groups. We started with some diagrammatic calculus, which was independently invented in both contexts. We studied it from the quantum group perspective and using quantum group tools we have proven something about quantum isomorphisms of certain graphs, which is again a topic that first came from quantum information. We hope that this work motivates further cooperation of these two communities.

The article is structured as follows. In Section~\ref{sec.prelim}, we recall preliminary information about diagrammatic categories. We also introduce the category $\NCBipartEven$, which lies at the centre of this work and study its basic properties. In Section~\ref{sec.Had}, we recall what a Hadamard matrix is and generalize this definition to the quantum setting and also to the purely categorical setting. In Section~\ref{sec.cat}, we study the structure of the category $\NCBipartEven$. Section~\ref{sec.qg} contains some preliminaries regarding quantum groups and quantum symmetries. In Section~\ref{sec.qgH}, we study the quantum groups related to $\NCBipartEven$ and show that they must be non-classical. Finally, in Section~\ref{sec.main} we show the main result of this article regarding quantum symmetries of Hadamard matrices. In the end, we mention some concluding remarks in Section~\ref{sec.conclusion}.

Readers that are mostly interested in details on the quantum isomorphisms of Hadamard matrices and graphs are advised to browse through Section~\ref{sec.prelim} and then skip directly to Section~\ref{sec.main}.

\section{Diagrammatic categories}
\label{sec.prelim}

Diagrammatic categories are the main tool for this article. Since this article is connecting three different areas of mathematics: quantum groups, quantum information theory, and diagrammatic categories, we decided to make this introductory section slightly more detailed. Actually, although all the mentioned results are known to experts, many of them are not stated anywhere in the literature in this form.

%In this section, we are going to summarize all preliminaries we need to motivate and formulate our results.

\subsection{Diagrammatic categories}

In this work, a \emph{category} will always be a rigid monoidal $\dagger$-category over $\C$ with the set of natural numbers $\N_0$ as the set of objects. Such a category $\Cat$ is called \emph{concrete} if it is realized by linear maps between finite-dimensional vector spaces, i.e.\ $\Cat(k,l)\subset\Lin((\C^N)^{\otimes k},(\C^N)^{\otimes l})$ for some $N\in\N$. We will denote by $\Mat$ the \emph{category of matrices}, i.e.\ the category with morphism spaces $\Mat(k,l)$ consisting of all linear maps $\C^k\to\C^l$ (i.e.\ the object 1 is identified with $\C$). In addition, we will denote by $\Mat_N$ the full subcategory of $\Mat$ given by $\Mat_N(k,l)=\Lin((\C^N)^{\otimes k},(\C^N)^{\otimes l})\simeq\Mat(N^k,N^l)$.

Loosely speaking a concrete category is a set of matrices $\Cat$ which is closed under matrix multiplication (whenever possible), tensor product, conjugate transposition (denoted by $\dag$) and which contains certain \emph{duality morphisms}. In contrast, an abstract category is just a set of abstract vector spaces $\Cat(k,l)$ equipped with some abstract multiplication rules (called composition and tensor product) and an involution~$\dag$ satisfying some axioms. We may then look for concrete realizations of the abstract category as concrete categories. Such a realization -- that is, a functor $F\colon\Cat\to\Mat$ -- is then called a \emph{fibre functor}.

Specifying a number $N\in\N$, we can start with a couple of linear maps $A,B,C,\dots$ between some tensor powers of $\C^N$ and ask what is the smallest category $\Cat\subset\Mat_N$ containing those. We will denote this category by $\langle A,B,C,\dots\rangle_N$. 

Now it is convenient to illustrate the elements of such a category using pictures. The generators $A$, $B$, $C,\dots$ are denoted by some boxes and the copies of the vector space $\C^N$ are denoted by strings. For instance, if $A\colon \C^N\to \C^N\otimes\C^N$, we may denote it by $\Graph{\draw (0.7,0.5) -- (0.7,1.5); \draw (1.3,0.5) -- (1.3,1.5);\draw (1,-0.5) -- (1,0.5) node[fill=white, rectangle]{$\scriptstyle A$};}$ (all diagrams are to be read from bottom to top\footnote{There is unfortunately no agreement in the literature on in which direction the diagrams should be read. In the theory of diagrammatic categories related to quantum groups, the diagrams are usually drawn from top to bottom. But since we want to make friends in the categorical QIT community, we will draw everything from bottom to top. Nevertheless, also left to right or right to left directions can be found in the literature.}). On the other hand, if $B\colon \C^N\otimes\C^N\to\C^N$, then it can be drawn as $\Graph{\draw (0.7,0.5) -- (0.7,-0.5); \draw (1.3,0.5) -- (1.3,-0.5);\draw (1,1.5) -- (1,0.5) node[fill=white, rectangle]{$\scriptstyle B$};}$. Finally, we can compute the composition $AB$ or $BA$, whose diagrams are $\Graph{\draw (0.7,1) -- (0.7,2); \draw (1.3,1) -- (1.3,2);\draw (0.7,0) -- (0.7,-1); \draw (1.3,0) -- (1.3,-1);\draw (1,0) node[fill=white,rectangle]{$\scriptstyle B$} -- (1,1) node[fill=white, rectangle]{$\scriptstyle A$};}$ and $\Graph{\draw (0.7,0)--(0.7,1);\draw (1.3,0) -- (1.3,1);\draw (1,1) node[fill=white, rectangle]{$\scriptstyle B$} -- (1,2);\draw (1,0) node[fill=white,rectangle]{$\scriptstyle A$} -- (1,-1);}$, respectively, or their tensor product $A\otimes B$ or $B\otimes A$ with diagrams $\Graph{\draw (0.7,0.5) -- (0.7,1.5); \draw (1.3,0.5) -- (1.3,1.5);\draw (1,-0.5) -- (1,0.5) node[fill=white, rectangle]{$\scriptstyle A$};}\Graph{\draw (0.7,0.5) -- (0.7,-0.5); \draw (1.3,0.5) -- (1.3,-0.5);\draw (1,1.5) -- (1,0.5) node[fill=white, rectangle]{$\scriptstyle B$};}$, $\Graph{\draw (0.7,0.5) -- (0.7,-0.5); \draw (1.3,0.5) -- (1.3,-0.5);\draw (1,1.5) -- (1,0.5) node[fill=white, rectangle]{$\scriptstyle B$};}\Graph{\draw (0.7,0.5) -- (0.7,1.5); \draw (1.3,0.5) -- (1.3,1.5);\draw (1,-0.5) -- (1,0.5) node[fill=white, rectangle]{$\scriptstyle A$};}$. In the end, every element of the category $\langle A,B,C,\dots\rangle_N$ is made by a finite amount of compositions, tensor products, and linear combinations from the generators and their involutions and hence every element can be represented by a linear combination of such pictures.

As we mentioned already, we want to deal with rigid categories. This means that the category $\Cat$ contains (usually among its generators) a \emph{duality morphism} $R\colon\C\to\C^N\otimes\C^N$. We usually use the diagrammatic notation of a \emph{cup} $R=\uppairpart$ and a \emph{cap}\footnote{This is actually a bit more restrictive than necessary. For rigidity, we need duality morphisms $R=\uppairpart$, $R\Ctrans=\pairpart$ satisfying the snake equation. But there is in general no reason to assume $R\Ctrans=R^\dag$. Nevertheless, all categories mentioned in this article will satisfy this additional condition.} $R^\dag=\pairpart$. We will use this cup and cap notation also in the case of abstract diagrammatic categories. As an axiom, the duality morphism satisfies the \emph{snake equation}
$\Graph{\draw (1,0) -- (1,1) -- (2,1) -- (2,0) -- (3,0) -- (3,1);}=
\idpart=
\Graph{\draw (1,1) -- (1,0) -- (2,0) -- (2,1) -- (3,1) -- (3,0);}$.
In the concrete categories, the duality morphism will usually be given by
\begin{equation}\label{eq.R}
R=\uppairpart=\sum_{i=1}^n e_i\otimes e_i.
\end{equation}
Note that this concrete $R$ is, in addition, \emph{symmetric}, so $\Graph{\draw (1,1) -- (2,0.3) -- (2,0) -- (1,0) -- (1,0.3) -- (2,1);}=\uppairpart$. The duality morphism $R\in\Cat(0,2)$ already induces a duality morphism $R_k\in\Cat(0,2k)$ for arbitrary $k$ by $R_k=
\Graph{
	\draw (1,1) -- (1,0) -- (6,0) -- (6,1);
	\draw (1.5,1) -- (1.5,0.2) -- (5.5,0.2) -- (5.5,1);
	\draw (3,1) -- (3,0.5) -- (4,0.5) -- (4,1);
	\node[draw=none,fill=none] at (2.25,.8) {$\scriptstyle\dots$};
	\node[draw=none,fill=none] at (4.75,.8) {$\scriptstyle\dots$};
}
=
(\id_{k-1}\otimes R\otimes\id_{k-1})(\id_{k-2}\otimes R\otimes\id_{k-2})\cdots R
$.

Given any rigid category $\Cat$ and a morphism $A\in\Cat(1,1)$, we can define its (left) transposition or (left) trace as
$$A\Ctrans=
\Graph{\draw (1,-.5) -- (1,1.5) -- (2,1.5) -- (2,0.5) node[fill=white,rectangle]{$\scriptstyle A$} -- (2,-.5) -- (3,-.5) -- (3,1.5);}
,\qquad
\CTr A=
\Graph{\draw (1,-.5) -- (1,0.5) node[fill=white,rectangle]{$\scriptstyle A$} -- (1,1.5) -- (2,1.5) -- (2,-.5)--cycle;}.
$$
If $\Cat$ is a concrete category and the duality morphism is given by \eqref{eq.R}, this resembles the standard notion of a matrix transposition and trace, which will be denoted by $A^{\rm T}$, $\Tr A$. On the other hand, if the duality morphism is not symmetric, then these notions actually do not satisfy the expected properties (transposition is not involutive, trace is not tracial and so on). The definition can be generalized to arbitrary element $A\in\Cat(k,l)$ using the duality morphisms $R_k$, $R_l$ instead of the simple cups and caps (we need $k=l$ for the trace).

Choosing the generators in a convenient way, we may be able to formulate some reduction rules, which then give rise to a \emph{diagrammatic calculus}. The best case scenario happens if the diagrammatic calculus is powerful enough such that the reduced diagrams form a basis of our category. However, this does not happen very often, so we usually require a weaker condition, see Def.~\ref{D.pure}. But first, a concrete example comes in handy.

\subsection{Spiders and partitions}

For $k,l\in\N_0$, we define the diagram $\spider$ with $k$ inputs and $l$ outputs to denote the tensor $T_{k,l}\colon (\C^N)^{\otimes k}\to(\C^N)^{\otimes l}$ whose entries are given by the Kronecker delta: $[T_{k,l}]_{i_1\cdots i_k}^{j_1\cdots j_l}=\delta_{i_1\cdots i_kj_1\cdots j_l}$. For $k=l=0$, we define $T_{0,0}=N$. We will call these morphisms \emph{black spiders}.

First, note that $\Gpair$ and $\Guppair$ satisfy the snake equation $\Graph{\draw (1,0) -- (1,0.8) -- (1.5,1) node{} -- (2,0.8) -- (2,0.2) -- (2.5,0) node{} -- (3,0.2) -- (3,1);}=\idpart=\Graph{\draw (1,1) -- (1,0.2) -- (1.5,0)node{} -- (2,0.2) -- (2,0.8) -- (2.5,1)node{} -- (3,0.8) -- (3,0);}$, so it can be used as a duality morphism. In fact, it actually coincides with the duality morphism from Eq.~\eqref{eq.R}. So, we will denote $\pairpart:=\Gpair$ and $\uppairpart:=\Guppair$.

Secondly, black spiders are symmetric, closed under the involution and can be fused together. That is, they satisfy the following reduction rules:
\begin{equation}\label{eq.black}
\Graph{\node at (1,0.5){};}=N,\qquad
\Big(\underbrace{\overbrace{
\Graph{
	\draw (3,0) -- (3,0.5) node{} -- (3,1);
	\draw (1,0) -- (1,0.2) -- (3,0.5);
	\draw (2,0) -- (2,0.2) -- (3,0.5);
	\draw (5,0) -- (5,0.2) -- (3,0.5);
	\draw (1,1) -- (1,0.8) -- (3,0.5);
	\draw (2,1) -- (2,0.8) -- (3,0.5);
	\draw (5,1) -- (5,0.8) -- (3,0.5);
	\node[draw=none,fill=none] at (4,0.1) {$\dots$};
	\node[draw=none,fill=none] at (4,0.9) {$\dots$};
}}^l}_k\Big)^\dag
=
\underbrace{\overbrace{\Graph{
	\draw (3,0) -- (3,0.5) node{} -- (3,1);
	\draw (1,0) -- (1,0.2) -- (3,0.5);
	\draw (2,0) -- (2,0.2) -- (3,0.5);
	\draw (5,0) -- (5,0.2) -- (3,0.5);
	\draw (1,1) -- (1,0.8) -- (3,0.5);
	\draw (2,1) -- (2,0.8) -- (3,0.5);
	\draw (5,1) -- (5,0.8) -- (3,0.5);
	\node[draw=none,fill=none] at (4,0.1) {$\dots$};
	\node[draw=none,fill=none] at (4,0.9) {$\dots$};
}}^k}_l,\qquad
\Graph{
	\draw (1,1.5) -- (1,1.3) -- (3,1) node{} -- (1, 0.7) -- (1,-0.5);
	\node[draw=none,fill=none,anchor=base] at (2,1.4) {$\overbrace{\dots}^{l_1}$};
	\draw (3,1.5) -- (3,1.3) -- (3,1);
	\node[draw=none,fill=none,anchor=base] at (2,-0.5) {$\underbrace{\dots}_{k_1}$};
	\draw (3,-0.5) -- (3,0.8) -- (3,1);
	\draw (7,1.5) -- (7,0.3) -- (7,0) node{} -- (7,-0.3) -- (7,-0.5);
	\node[draw=none,fill=none,anchor=base] at (8,1.4) {$\overbrace{\dots}^{l_2}$};
	\draw (9,1.5) -- (9,0.3) -- (7,0);
	\node[draw=none,fill=none,anchor=base] at (8,-0.5) {$\underbrace{\dots}_{k_2}$};
	\draw (9,-0.5) -- (9,-0.3) -- (7,0);
	\draw (3,1) -- (4,0.7) -- (4,0.3) -- (7,0);
	\draw (3,1) -- (6,0.7) -- (6,0.3) -- (7,0);
	\node[draw=none,fill=none,anchor=base] at (5,0.5) {$\overbrace{\vrule height 5pt width 0pt\dots}^m$};
}
=
\underbrace{\overbrace{\Graph{
	\draw (3,0) -- (3,0.5) node{} -- (3,1);
	\draw (1,0) -- (1,0.2) -- (3,0.5);
	\draw (2,0) -- (2,0.2) -- (3,0.5);
	\draw (5,0) -- (5,0.2) -- (3,0.5);
	\draw (1,1) -- (1,0.8) -- (3,0.5);
	\draw (2,1) -- (2,0.8) -- (3,0.5);
	\draw (5,1) -- (5,0.8) -- (3,0.5);
	\node[draw=none,fill=none] at (4,0.1) {$\dots$};
	\node[draw=none,fill=none] at (4,0.9) {$\dots$};
}}^{l_1+l_2}}_{k_1+k_2}
\end{equation}
\begin{equation}\label{eq.blacksym}
\Graph{
	\draw (1,-0.5) -- (1,0.2) -- (3.5,0.5) node{} -- (1,0.8) -- (1,1);
	\draw (4,-0.5) -- (4,-0.3) -- (3,0) -- (3,0.2) -- (3.5,0.5);
	\draw (3,-0.5) -- (3,-0.3) -- (4,0) -- (4,0.2) -- (3.5,0.5);
	\draw (6,-0.5) -- (6,0.2) -- (3.5,0.5);
	\draw (3,1) -- (3,0.8) -- (3.5,0.5);
	\draw (4,1) -- (4,0.8) -- (3.5,0.5);
	\draw (6,1) -- (6,0.8) -- (3.5,0.5);
	\node[draw=none,fill=none] at (2,-0.2) {$\dots$};
	\node[draw=none,fill=none] at (2,0.9) {$\dots$};
	\node[draw=none,fill=none] at (5,-0.2) {$\dots$};
	\node[draw=none,fill=none] at (5,0.9) {$\dots$};
}
=
\Graph{
	\draw (1,0) -- (1,0.2) -- (3.5,0.5) node{} -- (1,0.8) -- (1,1);
	\draw (3,0) -- (3,0.2) -- (3.5,0.5);
	\draw (4,0) -- (4,0.2) -- (3.5,0.5);
	\draw (6,0) -- (6,0.2) -- (3.5,0.5);
	\draw (3,1) -- (3,0.8) -- (3.5,0.5);
	\draw (4,1) -- (4,0.8) -- (3.5,0.5);
	\draw (6,1) -- (6,0.8) -- (3.5,0.5);
	\node[draw=none,fill=none] at (2,0.1) {$\dots$};
	\node[draw=none,fill=none] at (2,0.9) {$\dots$};
	\node[draw=none,fill=none] at (5,0.1) {$\dots$};
	\node[draw=none,fill=none] at (5,0.9) {$\dots$};
}
\end{equation}

As a consequence, any connected diagram made out of spiders equals to a single spider. More generally, any (possibly unconnected) diagram made out of spiders, i.e.\ any element $T\in\langle T_{k,l}\mid k,l\in\N_0\rangle_N$ can be reduced to a \emph{partition} of the $k$ inputs and $l$ outputs. In fact, strictly speaking, if we only allow tensor products, compositions, and involutions for spiders, we can never obtain a \emph{crossing} in our diagram, so we can obtain exactly all \emph{non-crossing partitions}. If we add the diagram $\crosspart$, then we can indeed obtain any partition.

\begin{defn}
Consider arbitrary $N\in\C$. The \emph{category of all partitions} $\Part_N$ is the category with morphism spaces spanned by diagrams of (possibly crossing) strings and black spiders subject to the relations \eqref{eq.black}, \eqref{eq.blacksym}. The \emph{category of all non-crossing partitions} $\NCPart_N$ is the category with morphism spaces spanned by diagrams of non-crossing strings and black spiders subject to the relations \eqref{eq.black}. For $N\in\N$, we will denote by $F_N$ the fibre functor $\Part_N\to\Mat_N$ interpreting the black spideres in the \emph{standard} way as described above: $F_N(\spider)=T_{k,l}$.
\end{defn}

Note that there may be a subtle difference between the abstract categories $\Part_N$ or $\NCPart_N$ and the concrete categories generated by the linear maps: the functor $F_N$ interpreting the categories may not be injective. That is, the associated linear maps may satisfy some additional relations that cannot be algebraically derived from the above mentioned ones.

\begin{defn}\label{D.pure}
A category $\Cat$ is called \emph{pure} if $\Cat(0,0)\simeq\C$.
\end{defn}

In case of diagrammatic categories, being pure means that all diagrams with no inputs and outputs can be reduced to a number (scalar multiple of an empty diagram). That is, there is no reduced diagram with no inputs and outputs. In contrast, concrete categories are always pure.

If a diagrammatic category is pure, it means that we essentially know all the relations.

\begin{prop}\label{P.pure}
Let $\Cat$ be a pure category. Then for every non-trivial fibre functor $F\colon \Cat\to\Mat$, we have $\ker F=\Neg$, where
\begin{equation}\label{eq.neg}
\Neg(k,l)=\{a\in\Cat(k,l)\mid \CTr(ab)=0 \text{ for every }b\in\Cat(l,k)\}
\end{equation}
is the tensor ideal of \emph{negligible morphisms}.
\end{prop}
\begin{proof}
Denote $A=F(a)$, $B=F(b)$. The assignment $(A,B)\mapsto F(\CTr(a^\dag b))$ is essentially the Hilbert--Schmidt inner product, which must be positive definite. (If the duality morphism $F(\uppairpart)$ is given by \eqref{eq.R}, then $F(\CTr(a^\dag b))=\Tr(A^\dag B)$, so it is exactly the Hilbert--Schmidt product. Nevertheless, it is surely positive definite in general as we have $\CTr(a^\dag b)=\tilde a^\dag\tilde b$, where $\tilde a=(a\otimes\id_k)R_k$, $\tilde b=(b\otimes\id_l)R_l$.) Hence, every fibre functor must map all negligible morphisms to zero. On the other hand if $F(a)=0$ for some $a\not\in\Neg$, then also $F(a^\dag a)=0$, where $a^\dag a$ is a non-zero element of $\Cat(0,0)\simeq\C$. But this means that $F$ maps everything to zero.  (In general, $\Neg$ is always the largest proper tensor ideal in $\Cat$ \cite{Bru00}. See also \cite{GW03}.)
\end{proof}

\begin{rem}
\label{R.FibreExistence}
Formula~\eqref{eq.neg} is the standard way how negligible morphisms are defined in the literature. But provided that some fibre functor on $\Cat$ actually exists, we can identify negligible morphisms in a simpler way as
$$\Neg(k,l)=\{a\in\Cat(k,l)\mid \CTr(a^\dag a)=0\}.$$
Indeed, since the Hilbert--Schmidt inner product is positive definite and hermitean, the sesquilinear form $\langle a,b\rangle=\CTr(a^\dag b)$ must be positive semidefinite and hermitean. Then by Cauchy--Schwarz inequality, we have that $\CTr(a^\dag a)=0$ implies $\CTr(ab)=0$ for any $b$.

Conversely, given a diagrammatic category $\Cat$, we can prove that there is no fibre functor on $\Cat$ if we show that the form is not positive semidefinite. For instance, using the results from \cite{Tut93,Jun19}, we can show this way that there is no fibre functor on $\NCPart_N$ unless $N=4\cos^2(\pi/l)$, $l\in\N$, $l\ge 3$ or $N\ge 4$.
\end{rem}

\begin{rem}
Both $\Part_N$ and $\NCPart_N$ are obviously pure. In fact, $\NCPart_N$ has non-trivial negligible morphisms only for $N=4\cos^2(j\pi/l)$, $j=1,\dots,l-1$, $l\in\N$ \cite[Prop.~3.12]{FM21} (see also \cite{Tut93,Jun19}). In particular, it has no negligible morphisms for $N\ge 4$. So, any fibre functor on $\NCPart_N$ is injective if $N\ge 4$.
\end{rem}

The category of all partitions was probably first defined in \cite{Mar94}. In \cite{Jon94} it was shown that it models the representation category of the symmetric group $S_N$. The name \emph{spider} comes from categorical quantum mechanics, more precisely the theory of \emph{ZX-calculus} invented by Coecke and Duncan \cite{CD07} (see \cite{CK17} for a detailed introduction). The non-crossing version is of a special interest in the theory of quantum groups \cite{BS09} as we are going to sketch in Section~\ref{sec.qg}.

Finally, it is worth mentioning two categories that are even simpler, but still of a great interest: If we forget about spiders and only consider diagrams with strings, we obtain the category of all pairings $\Pair_N$ known as the \emph{Brauer category} (Brauer showed that it models the representation category of the orthogonal group \cite{Bra37}). Even smaller is the category of all non-crossing pairings $\NCPair_N$, which is spanned by string diagrams, where the strings are not allowed to cross. It became famous under the name \emph{Temperley--Lieb category} \cite{TL71,Kau87}.

\subsection{Fibre functors on partitions}
Interesting question: Can we realize the categories $\Part_N$ or $\NCPart_N$ in a different way? That is, is there some alternative fibre functor $\Part_N\to\Mat$?

\begin{prop}\label{P.FibreNCPart1}
Consider $\delta\in\C$. There is a one-to-one correspondence between
\begin{enumerate}
\item fibre functors $F\colon\NCPart_{\delta^2}\to\Mat$,
\item special Frobenius $*$-algebras $\Alg$ with $\eta^\dag\eta=\delta^2$,
\item finite-dimensional C*-algebras $\Alg$ equipped with a $\delta$-form.
\end{enumerate}
\end{prop}
The equivalence $(2)\Leftrightarrow(3)$ will be immediately clear after we explain what a Frobenius algebra is. After that, we will prove the equivalence $(1)\Leftrightarrow(2)$. See \cite{Koc03} for a nice introduction to Frobenius algebras including the proof of various equivalent definitions.%  In order to prove this proposition, it is enough to clearly specify, what we mean by a special Frobenius $*$-algebra. There are two equivalent definitions common in the literature -- a concrete one and a categorical one. The first will make the equivalence $(2)\Leftrightarrow(3)$ obvious whereas the second will make the equivalence $(2)\Leftrightarrow(1)$ obvious.

\begin{defn}
A \emph{Frobenius algebra} is a finite-dimensional algebra $\Alg$ equipped with a linear functional $\psi$ such that the bilinear form $(a,b)\mapsto\psi(ab)$ is non-degenerate. Working over $\C$, $\Alg$ is called a \emph{Frobenius $*$-algebra} if $\Alg$ is a $*$-algebra and $\psi$ is positive (i.e.\ $\psi(a^*a)\ge 0$ for every $a$).
\end{defn}

Any Frobenius $*$-algebra is equipped with an inner product $\langle a,b\rangle=\psi(a^*b)$. Since it acts on itself by left multiplication, it must actually be a C*-algebra. We will denote by $\dag$ the adjoint of any map $T\colon \Alg^{\otimes k}\to \Alg^{\otimes l}$ with respect to this inner product. Note that $\psi=\eta^\dag$ in that case, where $\eta\colon\C\to \Alg$ is the inclusion of the unit $1\mapsto 1_\Alg$. In the following text, we will also denote by $m\colon\Alg\otimes\Alg\to\Alg$ the multiplication map $m(a\otimes b)=ab$.

\begin{defn}
A Frobenius $*$-algebra\footnote{In the definition of \emph{special}, the dagger can be equivalently replaced by transposition (using the bilinear form) in which case the definition makes sense for arbitrary Frobenius algebras (the $*$-structure is not necessary).} is called
\begin{itemize}
\item \emph{special} if $mm^\dag=\id$,
\item \emph{symmetric} if $\psi$ is tracial i.e.\ if the bilinear form is symmetric.
\end{itemize}
\end{defn}

The notion of a \emph{$\delta$-form} was defined in \cite{Ban02} independently of the theory of Frobenius algebras, but it is indeed essentially the same thing as the special Frobenius structure:

\begin{defn}[\cite{Ban02}]
Let $\Alg$ be a finite-dimensional C*-algebra. A $\delta$-form on $\Alg$ is any state $\psi$ such that in the associated GNS Hilbert space (i.e.\ equipping $\Alg$ with the inner product $\langle a,b\rangle=\psi(a^*b)$) we have $mm^\dag=\delta^2\,\id$.
\end{defn}

Indeed, the subtle difference lies only in the normalization of $\psi$: For a $\delta$-form, we require that it is a \emph{state}, i.e.\ $\|\psi_\text{B}\|=1$ or, equivalently, $\psi_\text{B}(1_\Alg)=\eta^{\dag_\text{B}}\eta=1$, which fixes the normalization of $\psi$ and then we require $mm^{\dag_\text{B}}=\delta^2\id$ for some $\delta$. In contrast, for a special Frobenius algebra, we fix the normalization of $\psi$ by requiring $mm^\dag=\id$ without the $\delta^2$ factor, which in turn means that $\psi$ is cannot be a state. That is, take $\psi_\text{F}=\delta^2\psi_\text{B}$ (so $\eta^{\dag_\text{F}}=\eta^{\dag_\text{B}}$ and hence $\|\psi_\text{F}\|=\eta^{\dag_\text{F}}\eta=\delta^2$). Then one can check that $m^{\dag_\text{F}}=\frac{1}{\delta^2}m^{\dag_\text{B}}$, so $mm^{\dag_\text{F}}=1$. This finishes the proof $(2)\Leftrightarrow(3)$.

Now we have a look on the equivalence $(1)\Leftrightarrow(2)$. First, we observe the following:

%\begin{lem}
%Let $A$ be a Frobenius $*$-algbebra. Denote by $R\in A\otimes A$ the adjoint of the bilinear form and by $m\colon A\otimes A\to A$ the multiplication on $A$ (so, $R^\dag=\eta^\dag\circ m$).
%\begin{enumerate}
%\renewcommand{\theenumi}{\alph{enumi}}
%\item The pair $(R,R^\dag)$ satisfies the snake equation, i.e.\ $(\id\otimes R^\dag)(R\otimes\id)=\id=(R^\dag\otimes\id)(\id\otimes R)$.
%\item The multiplication $m$ is self-conjugated, i.e. $m^\dag=m\Ctrans=(R^\dag\otimes\id\otimes\id)(\id\otimes m\otimes\id\otimes\id)(\id\otimes\id\otimes R\otimes\id\otimes\id)(\id\otimes R)$
%\end{enumerate}
%\end{lem}
\begin{lem}
\label{L.Rm}
Let $\Alg$ be a Frobenius $*$-algbebra. Denote by $\pairpart\colon\Alg\otimes\Alg\to\C$ the associated bilinear form and by $\uppairpart$ its adjoint. Denote also by $\Gfork\colon\Alg\otimes\Alg\to\Alg$ the multiplication on $\Alg$.
\begin{enumerate}
\renewcommand{\theenumi}{\alph{enumi}}
\item The pair $(\uppairpart,\pairpart)$ satisfies the snake equation.
\item The multiplication on $\Alg$ is self-conjugate; more precisely $(\Gfork)^\dag=
\Graph{
\draw (4,-.2) -- (4,1.2) -- (2.5,1.2) -- (2.5,.5) node{} -- (3,.2) -- (3,-.2) -- (0,-.2) -- (0,1.2);
\draw (2.5,.5) -- (2,.2) -- (2,0) -- (1,0) -- (1,1.2);
}$
\end{enumerate}
\end{lem}
\begin{proof}
Denote by $R\colon\C\to\Alg\otimes\Alg$ the adjoint of the bilinear form and by $m\colon\Alg\otimes\Alg\to\Alg$ the multiplication on $\Alg$. Let $(e_i)$ be some orthonormal basis of $\Alg$ and denote by $R^{ij}$, $m_{ij}^k$ the tensor entries of $R$ and $m$ in this basis. First, we claim that $e_i^*=\sum_{j}R^{ij}e_j$. Indeed, since the basis is orthonormal, we can compute the coordinates of $e_i^*$ as $\langle e_j,e_i^*\rangle=\overline{\langle e_i^*,e_j\rangle}=\overline{\psi(e_ie_j)}=\overline{R^\dag(e_i\otimes e_j)}=R^{ij}$.

Now, the statement (a) follows by the fact that $*$ is involutive: $e_i=e_i^{**}=\left(\sum_jR^{ij}e_j\right)^*=\sum_{jk}\bar R^{ij}R^{jk}e_k$, so $\sum_j(R^\dag)_{ij}R^{jk}=\delta_{ik}$, which is exactly the equality $\Graph{\draw (1,0) -- (1,1) -- (2,1) -- (2,0) -- (3,0) -- (3,1);}=\idpart$. The second snake equation is then just a complex conjugate of the first one.

The statement (b) follows by the fact that $*$ is an antihomomorphism:
\begin{align*}
(e_ie_j)^*&=\left(\sum_k m_{ij}^ke_k\right)^*=\sum_{kl}(m^\dag)^{ij}_kR^{kl}e_l,\\
=e_j^*e_i^*&=\sum_{a,b,l}R^{ja}R^{ib}m^l_{ab}e_l.
\end{align*}

If we denote $\Gmerge:=(\Gfork)^\dag$, the equality above can be written as
$\Graph{
\draw (1.5,0.5) -- (2,0.7) -- (2,1);
\draw (3,1) -- (3,0) -- (1.5,0) -- (1.5,0.5) node{} -- (1,0.7) -- (1,1);
}
=
\Graph{
\draw (1.5,0.5) -- (2,0.3) -- (2,0) -- (-1,0) -- (-1,1);
\draw (1.5,1) -- (1.5,0.5) node {} -- (1,0.3) -- (1,.1) -- (0,.1) -- (0,1);
}$. Using the snake equation, it is straightforward to derive the claimed equality.
\end{proof}

Now, let us finally formulate the proof of the equivalence $(1)\Leftrightarrow(2)$ itself.

\begin{proof}
We start with the direction $(1)\rightarrow(2)$. So, let $F$ be some fibre functor $\NCPart_{\delta^2}\to\Mat$. Denote by $N$ the dimension of the object $\Alg:=F(1)$. We define the structure of a Frobenius $*$-algebra on $\Alg$ by taking the multiplication $m:=F(\Gfork)$ and the linear functional $\psi:=F(\Gsing)$. From the diagrammatic rules, it follows that $m$ is indeed associative and that it has the unit $F(\Gupsing)$. For any $a\in \Alg$, we define $a^*$ by transposition of $a^\dag$, i.e.\ $a^*=(a^\dag\otimes\id)R$, where $R=F(\uppairpart)$. Then one can easily derive that $\psi(a^*a)=a^\dag a\ge0$.

For the other direction, start with a Frobenius $*$-algebra $\Alg$ and denote by $m\colon \Alg\otimes \Alg\to \Alg$ the multiplication on $\Alg$ and by $\eta\colon\C\to \Alg$ the inclusion of the unit in $\Alg$. We also denote $R:=m^\dag\eta$, so $R^\dag:=\eta^\dag m$ is the associated bilinear form. We associate the following diagrams 
\begin{align*}
    m&=:\Gfork   &   m^\dag&=:\Gmerge\\
 \eta&=:\Gupsing &\eta^\dag&=:\Gsing\\
    R&=m^\dag \eta=
\Graph{
\draw (1.5,0) node{} -- (1.5,0.5);
\draw (1,1) -- (1,0.7) -- (1.5,0.5) node {} -- (2,0.7) -- (2,1);
}
=:
\Guppair
& R^\dag&=\eta^\dag m=
\Graph{
\draw (1.5,1) node{} -- (1.5,0.5);
\draw (1,0) -- (1,0.3) -- (1.5,0.5) node {} -- (2,0.3) -- (2,0);
}=:\Gpair
\end{align*}

Actually, we can interpret any spider by
$\spider=m_l^{\dag}m_k$, where $m_k$ denotes the $k$-fold product (which is well defined by associativity). Now, one needs to prove that all the reduction rules \eqref{eq.black} are satisfied. We will skip this part here as it would prolong the article inadequately. It is described in a very detailed manner e.g.\ in \cite{Koc03}. We should maybe just comment on the $*$/$\dag$ structure, which is not discussed in \cite{Koc03}. First, we have to show that the duality morphisms are adjoint of each other, that is, prove that $(R,R^\dag)$ indeed satisfy the snake equation. But we did this already in Lemma~\ref{L.Rm}(a). Secondly, we have to prove the reduction rule involving the dagger. This can be done using Lemma~\ref{L.Rm}(b).
\end{proof}

\begin{rem}
For the equivalence $(1)\Leftrightarrow(3)$, see also \cite[Thm.~1]{Ban02}.
\end{rem}

%\begin{rem}
% let us just mention that apart from associativity of $m$ (and hence coassociativity of $m^\dag$) a crucial step is to prove the so-called Frobenius law \cite[Lemma~2.3.19]{Koc03}
%\begin{equation}\label{eq.Flaw}
%\Graph{
%\draw (1,0) -- (1,0.8) -- (1.5,1);
%\draw (2.5,1.5) -- (2.5,0.7) -- (2,0.5);
%\draw (1.5,1.5) -- (1.5,1) node{} -- (2,0.5) node{} -- (2,0);
%}
%=
%\Graph{
%\draw (1.5,0.5) -- (1.5,1);
%\draw (1,0) -- (1,0.3) -- (1.5,0.5) node {} -- (2,0.3) -- (2,0);
%\draw (1,1.5) -- (1,1.2) -- (1.5,1) node {} -- (2,1.2) -- (2,1.5);
%}
%=
%\Graph{
%\draw (2.5,0) -- (2.5,0.8) -- (2,1);
%\draw (1,1.5) -- (1,0.7) -- (1.5,0.5);
%\draw (2,1.5) -- (2,1) node{} -- (1.5,0.5) node{} -- (1.5,0);
%}.
%\qedhere
%\end{equation}
%\end{rem}

\begin{rem}\label{R.delta}
Every finite-dimensional C*-algebra $\Alg$ (and hence any Frobenius $*$-algebra) can be decomposed as $\Alg=\bigoplus_i M_{n_i}(\C)$. It is well known that any state $\psi$ on $\Alg$ can be expressed as $\psi(a)=\Tr(Qa)$ for some $Q\in \Alg$. Denote $Q=\bigoplus_i Q_i$ according to the decomposition. Then $\psi$ is a $\delta$-form if and only if $\Tr(Q_i^{-1})=\delta^2$ for every $i$ \cite{Ban02}.
\end{rem}

\begin{rem}
There is also a very abstract categorical definition of Frobenius algebras. A \emph{Frobenius monoid} is an object $M$ in an abstract rigid monoidal category such that certain abstract morphisms exist (namely the multiplication $m$, the unit $\eta$, the comultiplication $d$, and the counit $\psi$) satisfying some relations (namely the associativity of $m$, unitality of $\eta$, coassociativity of $d$ counitality of $\psi$, and the so-called Frobenius law). Our definition of a Frobenius algebra then corresponds to a Frobenius monoid in $\Mat$ (or $\Hilb$, the category of finite-dimensional Hilbert spaces).

Adding the $*$-structure in such an abstract setting amounts to requiring that the abstract category is a $\dag$-category. (Which we assume in our article by default.)
\end{rem}

We are often interested in the case, where the duality morphisms are symmetric:
\begin{prop}\label{P.FibreNCPart2}
Consider $\delta\in\C$. There is a one-to-one correspondence between
\begin{enumerate}
\item fibre functors $F\colon\NCPart_{\delta^2}\to\Mat$ such that $F(\pairpart)$ is symmetric,
\item symmetric Frobenius algebras $\Alg$ with $\dim \Alg=\delta^2$,
\item C*-algebras $\Alg$ with $\dim \Alg=\delta^2$.
\end{enumerate}
In particular, such a fibre functor exists if and only if $\delta^2=N\in\N_0$.
\end{prop}
\begin{proof}
The condition that $F(\pairpart)$ is symmetric is obviously equivalent to saying that the associated Frobenius algebra $\Alg$ is symmetric or that the associated $\delta$-form on the C*-algebra $\Alg$ is tracial. Consider the decomposition of $\Alg$ from Remark~\ref{R.delta}. It is well known that there is a unique trace on every matrix algebra. This means that we may take $Q_i=\frac{n_i}{\delta^2}\,\id$. So, there is actually a unique tracial $\delta$-form on $\Alg$ given by $\psi(a)=\sum_i \frac{n_i}{\delta^2}\Tr_i(a)$. Since $1=\psi(1_\Alg)=\sum_i\frac{n_i^2}{\delta^2}$, we have that $\delta^2=\sum_i n_i^2=\dim \Alg$. See also \cite[Prop~2.1]{Ban99}.
\end{proof}

\begin{rem}
The equivalence $(2)\Leftrightarrow(3)$, where Frobenius algebras are taken in the abstract categorial setting was recently formulated in \cite{Vic11}.
\end{rem}

\begin{rem}\label{R.Mn}
In the case when $\Alg$ is the matrix algebra $\Alg=M_n(\C)$, the corresponding fibre functor $F$ has an interesting diagrammatic interpretation. In this case, we have $\eta^\dag=n\Tr$, the associated inner product is then given by $\langle e_{ij},e_{kl}\rangle=n\delta_{ik}\delta_{jl}$, so we have an orthonormal basis $(\sqrt n\,e_{ij})_{i,j=1}^n$. Hence, $M_n(\C)$ can be identified with $C^n\otimes\C^n$ by $\sqrt{n}e_{ij}\mapsto e_i\otimes e_j$, which provides a convenient diagrammatic description of the (co)multiplication and  (co)unit:
$$
m=\frac{1}{\sqrt n}\,\Graph{
\draw (0.9,0) .. controls (0.9,0.5) and (1.4,0.5) .. (1.4,1);
\draw (2.1,0) .. controls (2.1,0.5) and (1.6,0.5) .. (1.6,1);
\draw (1.1,0) .. controls (1.1,0.5) and (1.9,0.5) .. (1.9,0);
},\quad
m^\dag=\frac{1}{\sqrt n}\,\Graph{
\draw (0.9,1) .. controls (0.9,0.5) and (1.4,0.5) .. (1.4,0);
\draw (2.1,1) .. controls (2.1,0.5) and (1.6,0.5) .. (1.6,0);
\draw (1.1,1) .. controls (1.1,0.5) and (1.9,0.5) .. (1.9,1);
},\quad
\eta=\sqrt n\,\Graph{
\draw (0.9,1) .. controls (0.9,0.7) and (1.1,0.7) .. (1.1,1);
},\quad
\eta^\dag=\sqrt n\,\Graph{
\draw (0.9,0) .. controls (0.9,0.3) and (1.1,0.3) .. (1.1,0);
}.
$$

We can interpret this categorically as follows: Recall the Temperley--Lieb category of all non-crossing pairings $\NCPair_n(k,l)\subset\NCPart_n(k,l)$ consisting of all partitions, where every block has size two. Denote then by $\NCPair'_n$ the full subcategory of $\NCPair_n$ given by restricting to even objects only. Then $\NCPart_{n^2}$ is monoidally isomorphic to $\NCPair'_n$ through
$$
\Gfork\mapsto\frac{1}{\sqrt n}\,\Graph{
\draw (0.9,0) .. controls (0.9,0.5) and (1.4,0.5) .. (1.4,1);
\draw (2.1,0) .. controls (2.1,0.5) and (1.6,0.5) .. (1.6,1);
\draw (1.1,0) .. controls (1.1,0.5) and (1.9,0.5) .. (1.9,0);
},\quad
\Gmerge\mapsto\frac{1}{\sqrt n}\,\Graph{
\draw (0.9,1) .. controls (0.9,0.5) and (1.4,0.5) .. (1.4,0);
\draw (2.1,1) .. controls (2.1,0.5) and (1.6,0.5) .. (1.6,0);
\draw (1.1,1) .. controls (1.1,0.5) and (1.9,0.5) .. (1.9,1);
},\quad
\Gupsing\mapsto\sqrt n\,\Graph{
\draw (0.9,1) .. controls (0.9,0.7) and (1.1,0.7) .. (1.1,1);
},\quad
\Gsing\mapsto\sqrt n\,\Graph{
\draw (0.9,0) .. controls (0.9,0.3) and (1.1,0.3) .. (1.1,0);
}.
$$

See also \cite{KS08,GProj}.
\end{rem}

Finally, if we allow crossings, then from the relation \eqref{eq.blacksym} it follows that the Frobenius algebra / C*-algebra must actually be commutative, so we have the following.

\begin{prop}\label{P.FibrePart}
Any fibre functor $F\colon\Part_N\to\Mat$ such that $F(\crosspart)$ is the flip map is up to a change of basis given by the standard interpretation $F_N$.
\end{prop}
\begin{proof}
As we just said, if we add the crossing to our category and interpret it as the flip map, then the relation
$\Graph{
\draw (2,-.5) -- (1,0) -- (1,0.3) -- (1.5,0.5);
\draw (1,-.5) -- (2,0) -- (2,0.3) -- (1.5,0.5) node{} -- (1.5,1);
}=\Gfork$
implies that the associated C*-algebra $\Alg$ is commutative (and $N$-dimensional by Prop.~\ref{P.FibreNCPart2}). By Gelfand duality, this means that $\Alg\simeq C(X)$ for $X=\{1,\dots,n\}$. Denoting by $(\delta_i)$ the basis of canonical projections $\delta_i(j)=\delta_{ij}$, the multiplication is then given by $\delta_i\delta_j=\delta_{ij}\delta_i$, so $m_{ij}^k=\delta_{ijk}=[F_N(\Gfork)]_{ij}^k$. The unique tracial state is the normalized summation $\psi(f)=\frac{1}{N}\sum_if(i)$. Equivalently, the counit is given by the unnormalized summation $\eta^\dag(f)=\sum_if(i)$. That is $[\eta^\dag]_i=\eta^\dag(\delta_i)=1=[F_N(\Gsing)]_i$.
\end{proof}

\subsection{Complementary spiders}\label{secc.Bipart}
%Since diagrammatic (easy) subcategories of $\Part_N$ are classified, in order to find additional interesting diagrammatic categories, we need to extend the category $\Part_N$ somehow. In this article, we are going to study the situation, where we have two kinds of spiders:
%
Denoting by $(e_i)$ the standard basis of $\C^N$, the black spiders were defined by $T_{k,l}(e_{i_1}\otimes\cdots\otimes e_{i_k})=\delta_{i_1\cdots i_k}e_{i_1}^{\otimes l}$. Now take some other orthonormal basis $(f_i)$ and define $\tilde T_{k,l}(f_{i_1}\otimes\cdots\otimes f_{i_k})=\delta_{i_1\cdots i_k}f_{i_1}^{\otimes l}$ and denote these maps by \emph{white spiders} $\Graph{\draw (1,0) -- (1,0.2) -- (3,0.5); \draw (2,0) -- (2,0.2) -- (3,0.5); \draw (5,0) -- (5,0.2) -- (3,0.5); \draw (1,1) -- (1,0.8) -- (3,0.5); \draw (2,1) -- (2,0.8) -- (3,0.5); \draw (5,1) -- (5,0.8) -- (3,0.5); \node[draw=none,fill=none] at (4,0.1) {$\dots$}; \node[draw=none,fill=none] at (4,0.9) {$\dots$};\draw (3,0) -- (3,0.5) node[fill=white]{} -- (3,1);}$. Those will obviously satisfy the same relations. Now what happens if a black spider meets the white one?

A basis $(f_i)$ of $\C^N$ is called \emph{self-conjugate} if all the basis vectors have real entries, i.e.\ $\langle e_i,f_j\rangle\in\R$ for every $i,j=1,\dots,N$, where $(e_i)$ is the standard basis. Two orthonormal bases $(e_i)$ and $(f_j)$ are called \emph{mutually unbiased} if $|\langle e_i,f_j\rangle|=\frac{1}{\sqrt N}$ for every $i,j$. So, both conditions together mean that $\langle e_i,f_j\rangle=\pm\frac{1}{\sqrt N}$. It is straightforward to derive the following reduction rules for spiders corresponding to self-conjugate orthogonal mutually unbiased bases \cite[Theorem~9.40]{CK17}:
\begin{equation}\label{eq.bipart}
\uppairpart:=\Guppair={\Gwhite\Guppair},\qquad
\Graph{
	\draw[double] (1,0.2) -- (1,0.8);
	\draw (1,1.2) -- (1,0.8) node{};
	\draw (1,-0.2) -- (1,0.2) node[fill=white]{};
}=\frac{1}{N}\Graph{
	\draw (1,1.2) -- (1,0.8) node{};
	\draw (1,-0.2) -- (1,0.2) node[fill=white]{};
}
\end{equation}

We denote the corresponding abstract diagrammatic categories by
$$\NCBipart_N:=\langle\Gfork,{\Gwhite\Gfork},\pairpart\rangle,\qquad\Bipart_N:=\langle\Gfork,{\Gwhite\Gfork},\pairpart,\crosspart\rangle.$$
To be more precise, diagrams in $\NCBipart_N$ satisfy relations \eqref{eq.black} for both black and white spiders and relations \eqref{eq.bipart} for composing them together. In $\Bipart_N$, we have in addition the relation \eqref{eq.blacksym} for both black and white spiders.

Note again that for the sake of the definition of the abstract category, $N$ can be an arbitrary complex number distinct from zero. Black and white spiders satisfying relations~\eqref{eq.bipart} are sometimes called \emph{complementary} \cite[Def.~9.27]{CK17}. The reason for our notation $\Bipart$ will be clear in a moment. If $H$ is the corresponding transition matrix between bases $(e_i)$ and $(f_j)$ (alternatively, the associated Hadamard matrix, see Section~\ref{sec.Had}), we will denote by $F_H\colon\Bipart_N\to\Mat_N$ the corresponding fibre functor interpreting black and white spiders as described above.

So, what do the elements of these categories actually look like? Well, they are some black and white points connected by some strings. So, they are some graphs equipped with a (possibly defective) two-colouring of vertices and with some additional input/output strings. These inputs and outputs can be formalized as follows: a \emph{bilabelled graph} is a tuple $(G,(a_1,\dots,a_k),(b_1,\dots,b_l))$, where $G$ is a graph and $a_1,\dots,a_k,b_1,\dots,b_l$ are its vertices, where $a_1,\dots,a_k$ stand for the input strings and $b_1,\dots,b_l$ stand for the output strings, see \cite{MR20}.

In principle, the graphs can be arbitrary, even containing loops or multiple edges. But now comes a more important question: What are the reduced diagrams? From relations~\eqref{eq.black}, it follows that reduced diagrams should not contain an edge between two black or two white vertices. So, the two-colouring actually has to be proper (the graphs are actually bipartite). This also means that there are no loops. In addition, the relation \eqref{eq.bipart} means that all multiple edges can be reduced to either a simple edge or no edge. So, the graphs are actually simple. Finally, since we have $\uppairpart:=\Guppair={\Gwhite\Guppair}$, no vertex should have degree two (counting the output strings as well) and since $\Graph{\node at (1,0.5){};}=N=\Graph{\node[fill=white] at (1,0.5){};}$, no vertex should be isolated. On the other hand, it is easy to see that every graph satisfying these conditions is already reduced and can be constructed in $\Bipart_N$. In $\NCBipart_N$, we can only construct non-crossing diagrams (also called \emph{planar} bilabelled graphs, which is a bit stronger than just planarity of the underlying graph, again see \cite{MR20} for a proper definition). To summarize:

\begin{prop}
The elements of the category $\Bipart_N$ can be identified with two-coloured bilabelled graphs, where no vertex has degree zero or two in the above described sense. The category $\NCBipart_N$ can be identified with its subset containing planar bilabelled graphs only.
\end{prop}

Now, we can again ask about the fibre functors. In the crossing case, it is known that there are no others:

\begin{prop}\label{P.FibreBipart}
Every fibre functor $F\colon\Bipart_N\to\Mat$ such that $F(\crosspart)$ is the flip map is given by a pair of self-conjugate orthogonal complementary bases as described above.
\end{prop}
\begin{proof}
First of all, by Proposition~\ref{P.FibrePart}, the functor $F$ restricted to black spiders $\Part_N=\langle\Gfork,\pairpart\rangle\subset\Bipart_N$ must (up to change of basis) coincide with the standard interpretation $F_N$. The same must hold for the white spiders as $\langle\Gwhite\Gfork,\pairpart\rangle$ is also isomorphic to $\Part_N$. So, denote these two bases $(e_i)$ and $(f_j)$. Suppose that $(e_i)$ is actually the standard basis. It remains to show that $(f_j)$ is self-conjugate and that they are mutually unbiased. But this is true: Recall the notation $T_{k,l}:=\sum_{i=1}^N e_i^{\otimes l}\otimes e_i^{\dag\otimes k}$ and $\tilde T_{k,l}=\sum_{j=1}^n=f_j^{\otimes l}\otimes f_j^{\dag\otimes k}$ for the above described interpretation of black and white spiders. Note that the second relation of \eqref{eq.bipart} says that $T_{2,1}\tilde T_{1,2}=\frac{1}{N}T_{0,1}\tilde T_{1,0}$, so
$$\langle e_i,f_j\rangle^2=(e_i^\dag\otimes e_i^\dag)(f_j\otimes f_j)=e_i^\dag T_{2,1}\tilde T_{1,2}f_j=\frac{1}{N}e_i^\dag T_{0,1}\tilde T_{1,0}f_j=\frac{1}{N},$$
which is all we needed. See also \cite[Thm.~9.40]{CK17}
\end{proof}

Unfortunately, neither of these two categories is pure. That is, we still do not have enough reduction rules. Indeed, for instance the diagram $\Graph{\draw (1,0.5) node{} -- (2,0.5) node[fill=white]{};}$ cannot be further reduced. There are two possible solutions for this problem -- either add more relations or find a suitable subcategory for which the presented relations already are enough. In this work, we will study the second option.

\subsection{Two-coloured graphs with even degrees}
\label{secc.bipart}
We define the category
$$\NCBipartEven_N:=\langle\Gconnecter,{\Gwhite\Gconnecter},\pairpart\rangle\subset\NCBipart_N.$$

This category was recently introduced in \cite{GD4} in an attempt to define a free quantum version of Coxeter groups of type $D$.

\begin{prop}\label{P.NCBipartEven}
The category $\NCBipartEven_N$ can be identified with the set of bilabelled two-coloured graphs that are simple, planar, all vertices have even degree not equal to zero or two (counting the input/output strings as well).
\end{prop}
\begin{proof}
The extra condition is that all vertices have even degree, which comes simply from the fact that all generators have even degree. See \cite[Prop.~3.20]{GD4} for more details.
\end{proof}

\begin{prop}
The category $\NCBipartEven_N$ is pure.
\end{prop}
\begin{proof}
We need to show that every non-trivial diagram with no input/output strings can be further reduced. We do that by showing that every non-trivial planar bipartite graph where all vertices have even degree has at least one vertex of degree two:

Without loss of generality, assume that the graph is connected (and non-trivial). It is well known that planar connected bipartite graphs satisfy the inequality $e\le 2n-4$, where $e$ is the number of edges and $n$ is the number of vertices. Denoting by $\delta$ the minimal degree of the graph, we obviously have $e\ge\frac{1}{2}n\delta$. Consequently $\delta\le 4-8/n$, so $\delta\le 3$. But since we assume that all vertices have even degree, we actually must have $\delta=2$.
\end{proof}

As we already mentioned, the category was interpreted as a representation category of a certain quantum group for $N=4$ in \cite{GD4}. However, Proposition~\ref{P.FibreBipart} gives us many other fibre functors for arbitrary $N$. The goal of this article will be to interpret the corresponding quantum groups.

\section{Hadamard matrices and generalizations}
\label{sec.Had}

\subsection{Hadamard matrices}

Recall that an orthonormal basis $(f_i)$ is self-conjugate and mutually unbiased with the standard basis $(e_i)$ by definition if and only if $\langle e_i,f_j\rangle=\pm 1/\sqrt N$. Thus, multiplying the transition matrix by $\sqrt N$, we obtain the following:

\begin{defn}
\emph{Hadamard matrix} of order $N$ is an $N\times N$ matrix with $\pm 1$ entries $H\in M_N(\{\pm 1\})$ such that its rows (equivalently columns) are mutually orthogonal (i.e.\ $HH^\dag=N\,1_{\C^N}=H^\dag H$).
\end{defn}

\begin{ex}[Walsh matrices]
The following matrix
$$W_1=\begin{pmatrix}1&1\cr1&-1\end{pmatrix}$$
is a Hadamard matrix of size $2\times 2$. Now observe that if $A$ and $B$ are Hadamard matrices, then $A\otimes B$ is a Hadamard matrix. Consequently, we can construct a series of \emph{Walsh matrices} satisfying the Hadamard condition by $W_n=W_1^{\otimes n}$. (The result is a matrix of size $2^n\times 2^n$.)
\end{ex}

\begin{rem}
Actually, $W_1$ is the Fourier transform on $\Z_2$ and hence $W_n$ is the Fourier transform on $\Z_2^n$. These Hadamard matrices have the additional property that multiplying two rows (or columns) entrywise, we get another row (column), which gives the rows (columns) a group structure (namely $\Z_2^n$). General Hadamard matrices do not have this property. There is also a notion of complex Hadamard matrices, where the canonical example is the Fourier transform on arbitrary finite abelian group.
\end{rem}

\subsection{Hadamard morphisms}

Now, let us take the diagrammatic approach to Hadamard matrices. Given a Hadamard matrix $H$, we can denote it by the diagram~$\GHadamard$. The defining properties can be then expressed in the following way:
\begin{equation}\label{eq.HadMor}
\Graph{\draw (1,0) -- (1,0.5) \Gmapscr{} -- (1,1);}:=\Graph{\draw (1,0) -- (1,1) -- (1.5,1) -- (1.5,0.5) \Gmapsc{} -- (1.5,0) -- (2,0) -- (2,1);}=(\GHadamard)^\dag,\qquad
\Graph{
\draw (1.5,-0.5)--(1.5,-0.2)--(1,0.2)--(1,0.5) \Gmapsc{}--(1,0.8)--(1.5,1.2)--(1.5,1.5);
\draw (1.5,-0.2) node {}--(2,0.2)--(2,0.5) \Gmapsc{}--(2,0.8)--(1.5,1.2) node{};
}=
\Graph{\draw (1,0) -- (1,0.3) node{};\draw (1,1)--(1,0.7) node{};},\qquad
\Graph{\draw (1,0) -- (1,0.3) \Gmapsc{} -- (1,0.7) \Gmapscr{} -- (1,1);}
=N\,\idpart=
\Graph{\draw (1,1) -- (1,0.7) \Gmapsc{} -- (1,0.3) \Gmapscr{} -- (1,0);},
\end{equation}

We should probably explain, where the diagrams came from. The most straightforward is the last equation, which indeed just says $HH^\dag=N\,1_{\C^N}=H^\dag H$. The first equation says that $H^{\rm T}=H^\dag$, which equivalently means $H$ is self-conjugated $H=\bar H$, so it has real entries. For the middle one, note first that if black spiders are interpreted the standard way, then
$$
\Graph{
\draw (1,-1)--(1,-0.5)--(0,0)--(0,0.5) node[fill=white,rectangle]{$\scriptstyle A$}--(0,1)--(1,1.5)--(1,2);
\draw (1,-0.5) node {}--(2,0)--(2,0.5) node[fill=white,rectangle]{$\scriptstyle B$}--(2,1)--(1,1.5) node{};
}=A\bullet B,
$$
where $A\bullet B$ denotes the \emph{Schur product} (also known as the Hadamard product) defined entrywise. Hence the middle equation says that $H\bullet H=J$, where $J=\eta\eta^\dag$ is the all-one-matrix. Consequently, it means that the entries of $H$ are just $\pm 1$.

We can make all this abstract by defining the following diagrammatic categories.
$$\NCHad_N:=\langle\Gfork,\GHadamard,\pairpart\rangle,\qquad\Had_N:=\langle\Gfork,\GHadamard,\pairpart,\crosspart\rangle,$$
where the morphism $\GHadamard$ is supposed to satisfy equations~\eqref{eq.HadMor}. We call this morphism a \emph{Hadamard morphism}.

These categories are obviously not pure. For instance, the diagram $\Graph{\draw (1,0) -- (1,0.5) \Gmapsc{} -- (1,1) -- (1.5,1) -- (1.5,0) -- cycle;}$ cannot be reduced.\footnote{Hadamard matrices can indeed have different traces. For instance, all Walsh matrices have trace zero. On the other hand, matrices constructed by the so-called \emph{Payley construction of type I} have only $+1$ on the diagonal. See also the database of Hadamard matrices at \url{http://neilsloane.com/hadamard/}.} Nevertheless, they allow us to formulate an alternative approach to what we presented in Section~\ref{secc.bipart}.

\begin{prop}\label{P.FunctorHad}
There are functors $\NCBipart_N\to\NCHad_N$ and $\Bipart_N\to\Had_N$ acting trivially on black spiders and mapping ${\Gwhite\Gfork}\mapsto N^{-3/2}
\Graph{
\draw (1.5,0.5) -- (2,0.3) -- (2,0) \Gmapsc{} -- (2,-.3);
\draw (1.5,1.3) -- (1.5,1) \Gmapscr{} -- (1.5,0.5) node {} -- (1,0.3) -- (1,0) \Gmapsc{} -- (1,-.3);
}$.
\end{prop}
\begin{proof}
Since the categories $\NCBipart_N$ and $\Bipart_N$ are defined by generators and relations, it is enough to check that the same relations are satisfied by the images. First, since black spiders are mapped to black spiders, there is nothing to check regarding these relations. Secondly, we need to check that also the images of white spiders satisfy the relations for spiders. Here, it is straightforward to check that an image of any white spider is just a black spider of the same type with $N^{-1/2}\GHadamard$ connected to every input and $N^{-1/2}\Graph{\draw (1,0) -- (1,0.5) \Gmapscr{} -- (1,1);}$ connected to every output. Then, one can check that indeed all the relations are satisfied since when performing the composition, the extra normalized Hadamard morphisms cancel out. Finally, we need to check the compatibility conditions~\eqref{eq.bipart}. This is indeed also satisfied since
\[{\Gwhite\Guppair}\mapsto N^{-1}\Graph{\draw (1,1) -- (1,0.5) \Gmapscr{} -- (1,0) -- (2,0) -- (2,0.5) \Gmapscr{} -- (2,1);}=\uppairpart,\qquad
\Graph{
	\draw[double] (1,0.2) -- (1,0.8);
	\draw (1,1.2) -- (1,0.8) node{};
	\draw (1,-0.2) -- (1,0.2) node[fill=white]{};
}\mapsto
N^{-3/2}\Graph{
\draw (1.5,-0.8) -- (1.5,-0.6) \Gmapsc{} --(1.5,-0.2)--(1,0.2)--(1,0.5) \Gmapscr{}--(1,0.8)--(1.5,1.2)--(1.5,1.5);
\draw (1.5,-0.2) node {}--(2,0.2)--(2,0.5) \Gmapscr{}--(2,0.8)--(1.5,1.2) node{};
}=
N^{-3/2}\Graph{
	\draw (1,1.2) -- (1,0.8) node{};
	\draw (1,-0.5) -- (1,-0.2) \Gmapsc{} -- (1,0.2) node{};
}
\mathrel{\reflectbox{$\mapsto$}}
N^{-1}\Graph{
	\draw (1,1.2) -- (1,0.8) node{};
	\draw (1,-0.2) -- (1,0.2) node[fill=white]{};
}.
\qedhere\]
\end{proof}

\subsection{Quantum Hadamard matrices}
\label{secc.qhad}

We would like to study some additional fibre functors for $\NCBipart$ which do not extend to $\Bipart$. As already follows from Proposition~\ref{P.FibreNCPart1}, these will be based on some special Frobenius $*$-algebras. Note that special Frobenius $*$-algebras (or C*-algebras equipped with a $\delta$-form) are sometimes called \emph{finite quantum spaces}. We usually denote by $C(X)$ the $*$-algebra and by $l^2(X)$ the associated Hilbert space, where (as in case of quantum groups) $X$ denotes the abstract (non-existent) underlying quantum space. Note that some authors restrict only to symmetric Frobenius $*$-algebras.

Let $X$ be a quantum space and consider a linear map $A\colon l^2(X)\to l^2(X)$. Recall that we denote by $A^\dag$ its adjoint. We also denote by $A^*$ its \emph{conjugation}, i.e.\ the adjoint transposed
$$
A^*=(\id\otimes R^\dag)(\id\otimes A^\dag\otimes\id)(R\otimes\id)=
\Graph{\draw (1,-0.5) -- (1,1.3) -- (1.5,1.5) node{} -- (2,1.3) -- (2,0.5) node[fill=white,rectangle] {$\scriptstyle A^\dag$} -- (2,-0.3) -- (2.5,-0.5) node{} -- (3,-0.3) -- (3,1.5);}.
$$
We can generalize the Schur product of matrices to the setting of quantum spaces by
$$
A\bullet B=m(A\otimes B)m^\dag=
\Graph{
\draw (1,-1)--(1,-0.5)--(0,0)--(0,0.5) node[fill=white,rectangle]{$\scriptstyle A$}--(0,1)--(1,1.5)--(1,2);
\draw (1,-0.5) node {}--(2,0)--(2,0.5) node[fill=white,rectangle]{$\scriptstyle B$}--(2,1)--(1,1.5) node{};
},
$$
where in this case the spiders stand for the (co)multiplication in the Frobenius algebra.

\begin{defn}
Let $X$ be a quantum space with $\delta^2:=\eta^\dag\eta$. A \emph{quantum Hadamard matrix} is a linear map $H\colon l^2(X)\to l^2(X)$ such that 
$$H=H^*,\qquad H\bullet H=\eta\eta^\dag,\qquad HH^\dag=\delta^2\id=H^\dag H.$$
\end{defn}

Note that the defining relations exactly correspond to relations~\eqref{eq.HadMor}.

\begin{rem}
There are two approaches to constructing quantum analogues of classical matrices $A\colon\C^N\to\C^N$. If $A$ has some combinatorial flavour, i.e.\ can be seen as acting on a finite space $X=\{1,\dots,N\}$, we can replace this finite space $X$ by a finite quantum space and obtain $A\colon l^2(X)\to l^2(X)$. This is the approach we took here. In a similar way \emph{quantum graphs} are defined \cite{MRV18}. The second approach is keeping the underlying space $X$ or the vector space $\C^N$, but considering $A$ as a matrix with non-commutative entries, i.e.\ $A\in M_N(\C)\otimes\Alg$ for some C*-algebra $\Alg$. A typical example for this are \emph{quantum groups} (see Section~\ref{sec.qg}). Hadamard matrices with non-commutative entries were recently defined and studied by Banica in \cite{Ban18}.
\end{rem}

\begin{prop}\label{P.FibreHad}
There is a one-to-one correspondence between
\begin{enumerate}
\item fibre functors $F\colon\NCHad_{\delta^2}\to\Mat$,
\item finite quantum spaces $X$ with $\eta^\dag\eta=\delta^2$ equipped with a quantum Hadamard matrix $H$.
\end{enumerate}
\end{prop}
\begin{proof}
Directly from the definition.
\end{proof}

\begin{cor}\label{C.FibreNCBipart}
Any finite quantum space $X$ equipped with a quantum Hadamard matrix $H$ induces a fibre functor $F\colon\NCBipart_{\delta^2}\to\Mat$ by
$$F(\pairpart)=R^\dag,\qquad F(\Gfork)=m,\qquad F(\Gwhite\Gfork)=H^{-1}m(H\otimes H),$$
where $m$ is the multiplication in $C(X)$ and $R^\dag$ is the bilinear form on $C(X)$.
\end{cor}
\begin{proof}
The fibre functor is constructed by composing the fibre functor from Proposition~\ref{P.FibreHad} with the functor from Proposition~\ref{P.FunctorHad}
\end{proof}
%\begin{proof}
%Restricting to black spiders only, we get the category of non-crossing partitions $\NCPart_{\delta^2}\subset\NCBipart_{\delta^2}$, where the result follows from Proposition~\ref{P.FibreNCPart1}. Since the image of the white spider is only supposed to be conjugated by some invertible matrix $H$, everything must also work if we restrict to white spiders only. The only thing we need to show that the black and white spiders defined this way are indeed complementary, i.e.\ satisfy relations \eqref{eq.bipart}.
%
%This is indeed true. A diagrammatic proof can just be copied from the classical case....
%\end{proof}

Extending the notation from Section~\ref{secc.Bipart}, given a (quantum) Hadamard matrix $H$, we will denote by $F_H$ the both corresponding fibre functors $\NCHad_{\delta^2}\to\Mat$ and $\NCBipart_{\delta^2}\to\Mat$.

\begin{ex}\label{E.transpose}
Consider the finite quantum space $X=M_n$ from Remark~\ref{R.Mn}. That is, $C(X)=M_n(\C)$, the counit is given by $\eta^\dag=n\Tr$, so $\delta^2=n^2$. Now we claim that the map $H\colon M_n(\C)\to M_n(\C)$ acting by $H(a)=n\,a^{\rm T}$, where $a^{\rm T}$ is the matrix transposition, is a quantum Hadamard matrix. This is easy to check diagrammatically since given the identification $l^2(X)=\C^n\otimes\C^n$, we can write $H(e_i\otimes e_j)=n\,e_j\otimes e_i$, so $H=n\Graph{\draw (0.9,0) -- (1.1,1); \draw (1.1,0) -- (0.9,1);}$. Consequently, it is obvious that $H=H^*$ and we can easily check the other conditions:
$$\displaylines{
HH^\dag=n^2\Graph{\draw (0.9,-0.5) -- (1.1,0.5) -- (0.9,1.5); \draw (1.1,-0.5) -- (0.9,0.5) -- (1.1,1.5);}=n^2\Graph{\draw (0.9,0) -- (0.9,1);\draw (1.1,0) -- (1.1,1);}=n^2\Graph{\draw (0.9,-0.5) -- (1.1,0.5) -- (0.9,1.5); \draw (1.1,-0.5) -- (0.9,0.5) -- (1.1,1.5);}=H^\dag H\cr
H\bullet H=m(H\otimes H)m^\dag=n\Graph{
\draw (0.9,0) .. controls (0.9,-0.5) and (1.4,-0.5) .. (1.4,-1);
\draw (2.1,0) .. controls (2.1,-0.5) and (1.6,-0.5) .. (1.6,-1);
\draw (1.1,0) .. controls (1.1,-0.5) and (1.9,-0.5) .. (1.9,0);
\draw (0.9,0) -- (1.1,1); \draw (1.1,0) -- (0.9,1);
\draw (1.9,0) -- (2.1,1); \draw (2.1,0) -- (1.9,1);
\draw (0.9,1) .. controls (0.9,1.5) and (1.4,1.5) .. (1.4,2);
\draw (2.1,1) .. controls (2.1,1.5) and (1.6,1.5) .. (1.6,2);
\draw (1.1,1) .. controls (1.1,1.5) and (1.9,1.5) .. (1.9,1);
}=n\Graph{
\draw (0.9,1) .. controls (0.9,0.7) and (1.1,0.7) .. (1.1,1);
\draw (0.9,0) .. controls (0.9,0.3) and (1.1,0.3) .. (1.1,0);
}=\eta\eta^\dag
}$$
\end{ex}

\section{Structure of $\NCBipartEven$}
\label{sec.cat}

In \cite{GD4}, we conjectured that the number of morphisms in $\NCBipartEven_N(0,k)$ are given by $C_k^2$, where $C_k=\frac{1}{k+1}\binom{2k}{k}$ are \emph{Catalan numbers}. Recall that $C_k$ is the number of non-crossing partitions on $k$ points as well as the number of non-crossing pairings on $2k$ points, so $\dim\NCPart_{\delta^2}(0,k)=\dim\NCPair_{\delta^2}(0,2k)=C_k$. Motivated by this, we additionally conjectured in \cite{GD4} that $\NCBipartEven_N$ is isomorphic to $\NCPair_n\times\NCPair_n$ for appropriate $n$. We are going to prove this conjecture here. More precisely, we prove the following:

\begin{thm}\label{T.Catalan}
The category $\NCBipartEven_{\delta^2}$ is isomorphic (but not monoidally isomorphic) to $\NCPair_\delta\times\NCPair_\delta$ for every $\delta\neq 0$.
\end{thm}

First, we should probably make clear, what do we mean by the product $\times$.

\begin{defn}
Let $\Cat_1$ and $\Cat_2$ be two categories. We define $\Cat_1\times\Cat_2$ to be the category with morphism spaces $(\Cat_1\times\Cat_2)(k,l):=\Cat_1(k,l)\otimes\Cat_2(k,l)$. All operations are defined entrywise.
\end{defn}

See also the more general construction by Deligne \cite{Del90} and also the quantum group viewpoint in \cite{CW16}.

Note the following fact.
\begin{lem}\label{L.timesgen}
The category $\NCPair_\delta\times\NCPair_\delta$ is generated by $\Pabba\times\Paabb$, $\Paabb\times\Pabba$, and $\pairpart\times\pairpart$.
\end{lem}
\begin{proof}
Denote $\Cat:=\langle\Pabba\times\Paabb,\Paabb\times\Pabba,\pairpart\times\pairpart\rangle\subset\NCPair_\delta\times\NCPair_\delta$. We are trying to prove the inclusion $\NCPair_\delta\times\NCPair_\delta\subset\Cat$. By definition, $\Cat$ contains the duality morphism $\pairpart\times\pairpart$. Hence, by the so-called Frobenius reciprocity, it is enough to prove the equality for the endomorphism algebras $\Cat(k,k)$. It is known that the Temperley--Lieb algebras $\NCPair_\delta(k,k)$ are generated by elements of the form $e_i=\id_{i-1}\otimes\Paabb\otimes\id_{k-i-1}$, where $\id_j=\idpart^{\otimes j}$. It is clear that $e_i\times \id_k\in\Cat(k,k)$ for every $i,k$. Consequently, $p\times\id_k\in\Cat(k,k)$ for every $p\in\NCPair_\delta(k,k)$. Similarly, we have $\id_k\times q\in\Cat(k,k)$ for every $q\in\NCPair_\delta(k,k)$. Consequently, $p\times q=(\id_k\times q)(p\times \id_k)\in\Cat(k,k)$, which is what we wanted to show. See also \cite[Lemma~2.6]{CE23}.
\end{proof}

We claim that $\NCBipartEven_{\delta^2}$ is monoidally equivalent to a certain twist of $\NCPair_\delta\times\NCPair_\delta$. Namely consider the category $\NCPair_\delta\tiltimes\NCPair_\delta$, which is defined exactly the same way as $\NCPair_\delta\times\NCPair_\delta$ except for the tensor product of morphisms. Consider $T_i,S_i\in\NCPair_\delta(k_i,l_i)$, $i=1,2$. Note that $k_i$ must have the same parity as $l_i$ since otherwise the morphism space is actually empty. We define
$$(T_1\times S_1)\otimes(T_2\times S_2):=\begin{cases}
(T_1\otimes T_2)\times(S_1\otimes S_2)&\text{if $k_1,l_1$ are even,}\\
(T_1\otimes S_2)\times(S_1\otimes T_2)&\text{if $k_1,l_1$ are odd.}
\end{cases}
$$

\begin{lem}\label{L.twistgen}
The category $\NCPair_\delta\tiltimes\NCPair_\delta$ is generated by $\Pabba\times\Paabb$, $\Paabb\times\Pabba$, and $\pairpart\times\pairpart$.
\end{lem}
\begin{proof}
Follows from Lemma~\ref{L.timesgen} and the fact that the set of generators is invariant under swapping the $\times$-factors.
\end{proof}

Finally, denote by $\Pair_\delta\subset\Part_\delta$ the Brauer's category of all pairings and by $\Pair'_\delta$ its full subcategory given by restricting to even objects only (as in Remark~\ref{R.Mn}).

\begin{lem}\label{L.FNCBipart}
There is a monoidal unitary functor $\NCBipart_{\delta^2}\to\Pair'_\delta$ mapping
$$\Gfork\mapsto\frac{1}{\sqrt\delta}\TLfork,\quad\Gwhite\Gfork\mapsto\frac{1}{\sqrt\delta}\LTfork,\quad\pairpart\mapsto\TLpair$$
\end{lem}
\begin{proof}
It is well known (recall from Remark~\ref{R.Mn}) that there is a monoidal isomorphism $\NCPart_{\delta^2}\to\NCPair'_\delta\subset\Pair'_\delta$ given by $\Gfork\mapsto\frac{1}{\sqrt\delta}\TLfork$, $\pairpart\mapsto\TLpair$. From Example~\ref{E.transpose}, it follows that $\delta\,\Graph{\draw (0.9,0) -- (1.1,1); \draw (1.1,0) -- (0.9,1);}$ satisfies the relations \eqref{eq.HadMor} for a Hadamard morphism. Hence, we have a functor $\Had_{\delta^2}\to\Pair'_\delta$ mapping  $\GHadamard\mapsto\Graph{\draw (0.9,0) -- (1.1,1); \draw (1.1,0) -- (0.9,1);}$. Finally we compose it with the functor from Proposition~\ref{P.FunctorHad} in order to obtain the desired functor $\NCBipart_{\delta^2}\to\Pair'_\delta$.
\end{proof}

\begin{rem}\label{R.FNCBipart}
We can restrict this functor to $\NCBipartEven_{\delta^2}$ and obtain a functor
$$
\NCBipartEven_{\delta^2}=\langle\Gconnecter,\Gwhite\Gconnecter,\pairpart\rangle
\quad\to\quad
\langle\TLconnecter,\LTconnecter,\TLpair\rangle\subset\Pair_{\delta}'
$$
(On the right-hand side, we consider
$\langle\TLconnecter,\LTconnecter,\TLpair\rangle$
as a subcategory of $\Pair_\delta'$. So, in particular, it must contain the identity morphism $\TLid$, which we do not list explicitly as a generator.)
\end{rem}

\begin{lem}\label{L.FPair'}
The category $\langle\TLconnecter,\LTconnecter,\TLpair\rangle\subset\Pair'_\delta$ is monoidally isomorphic with $\NCPair_\delta\tiltimes\NCPair_\delta$
\end{lem}
\begin{proof}
For arbitrary diagram $t\in\langle\TLconnecter,\LTconnecter,\TLpair\rangle$ we are going to construct a pair $(p,q)$ of pair partitions. In order to do that, we colour all the upper and lower points by colours red, blue, blue, red, red, blue and so on.

We claim that in the category $\langle\TLconnecter,\LTconnecter,\TLpair\rangle$, inputs/outputs can be connected by a string only if they have the same colour. First, this is true for the generators. Secondly, we prove that this property is preserved under taking composition, tensor product, and involution. This is clear for involution as the colouring does not change after taking the horizontal flip. It is also clear for composition since, when composing, we clearly only connect strings of the same colour. It may be a bit unclear for tensor product. So, take two diagrams $t,s\in\langle\TLconnecter,\LTconnecter,\TLpair\rangle$. If the number of inputs (and therefore also outputs) of $t$ is divisible by four, then $t\otimes s$ is given just by writing the diagrams side by side and the colours match the rule, so if it is true that red (blue)
 points are only connected to red (blue) in $t$ and $s$, it must be true in $t\otimes s$ as well. If the number of inputs of $t$ is (divisible by two, but) not divisible by four (and hence the number of outputs as well), then the colouring of $t\otimes s$ is given by flipping the colours for $s$. But the same argument applies here as well.

Finally, we construct the two non-crossing pairings simply by restricting to the inputs/outputs with the red, resp.\ blue colour. Here is a couple of examples:
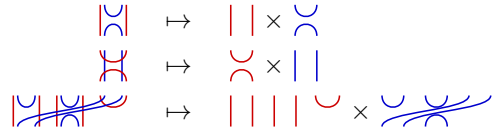
\begin{align*}
\Large\Graph{
\draw[red!80!black] (0.9,0) -- (0.9,1);
\draw[red!80!black] (2.1,0) -- (2.1,1);
\draw[blue!80!black] (1.1,0) .. controls (1.1,0.5) and (1.9,0.5) .. (1.9,0);
\draw[blue!80!black] (1.1,1) .. controls (1.1,0.5) and (1.9,0.5) .. (1.9,1);
}
&\quad\mapsto\quad
\Large\Graph{
\draw[red!80!black] (1,0) -- (1,1);
\draw[red!80!black] (2,0) -- (2,1);
}
\times
\Large\Graph{
\draw[blue!80!black] (1,0) .. controls (1,0.5) and (2,0.5) .. (2,0);
\draw[blue!80!black] (1,1) .. controls (1,0.5) and (2,0.5) .. (2,1);
}\\
\Large\Graph{
\draw[blue!80!black] (1.1,0) -- (1.1,1);
\draw[blue!80!black] (1.9,0) -- (1.9,1);
\draw[red!80!black] (0.9,0) .. controls (0.9,0.5) and (2.1,0.5) .. (2.1,0);
\draw[red!80!black] (0.9,1) .. controls (0.9,0.5) and (2.1,0.5) .. (2.1,1);
}
&\quad\mapsto\quad
\Large\Graph{
\draw[red!80!black] (1,0) .. controls (1,0.5) and (2,0.5) .. (2,0);
\draw[red!80!black] (1,1) .. controls (1,0.5) and (2,0.5) .. (2,1);
}\times
\Large\Graph{
\draw[blue!80!black] (1,0) -- (1,1);
\draw[blue!80!black] (2,0) -- (2,1);
}
\\
\Large\Graph{
\draw[red!80!black] (0.9,0) -- (0.9,1);
\draw[red!80!black] (2.1,0) -- (2.1,1);
\draw[blue!80!black] (1.1,1) .. controls (1.1,0.5) and (1.9,0.5) .. (1.9,1);
\draw[red!80!black] (2.9,0) -- (2.9,1);
\draw[red!80!black] (4.1,0) -- (4.1,1);
\draw[blue!80!black] (3.1,0) .. controls (3.1,0.5) and (3.9,0.5) .. (3.9,0);
\draw[blue!80!black] (3.1,1) .. controls (3.1,0.5) and (3.9,0.5) .. (3.9,1);
\draw[red!80!black] (4.9,1) .. controls (4.9,0.5) and (6.1,0.5) .. (6.1,1);
\draw[blue!80!black] (1.1,0) .. controls (1.1,0.5) and (5.1,0.5) .. (5.1,1);
\draw[blue!80!black] (1.9,0) .. controls (1.9,0.5) and (5.9,0.5) .. (5.9,1);
}
&\quad\mapsto\quad
\Large\Graph{
\draw[red!80!black] (1,0) -- (1,1);
\draw[red!80!black] (2,0) -- (2,1);
\draw[red!80!black] (3,0) -- (3,1);
\draw[red!80!black] (4,0) -- (4,1);
\draw[red!80!black] (4.9,1) .. controls (4.9,0.5) and (6,0.5) .. (6,1);
}
\times
\Large\Graph{
\draw[blue!80!black] (1,1) .. controls (1,0.5) and (2,0.5) .. (2,1);
\draw[blue!80!black] (3,0) .. controls (3,0.5) and (4,0.5) .. (4,0);
\draw[blue!80!black] (3,1) .. controls (3,0.5) and (4,0.5) .. (4,1);
\draw[blue!80!black] (1,0) .. controls (1,0.5) and (5,0.5) .. (5,1);
\draw[blue!80!black] (2,0) .. controls (2,0.5) and (6,0.5) .. (6,1);
}
\end{align*}

The map is clearly functorial and bijective for the reasons we said above (involution does not change colours, taking composition we connect reds to reds and blues to blues). The tensor product $t\otimes s$ works fine in the case when $t$ has the number of inputs divisible by four. If the number of inputs is not divisible by four, we need to switch red and blue in $s$, which precisely matches the defining property of the twisted tensor product.
\end{proof}

Finally, it remains to show that the functor from Remark~\ref{R.FNCBipart} is also an isomorphism. Let us first formulate a weaker version of Theorem~\ref{T.Catalan}, which will be much easier to prove.

\begin{prop}\label{P.Catalan}
The categories $\NCBipartEven_{n^2}$ and $\NCPair_n\tiltimes\NCPair_n$ are monoi\-dally isomorphic up to negligible morphisms for every $n\in\N$.
\end{prop}
\begin{proof}
We get a monoidal functor $F\colon\NCBipartEven_{n^2}\to\NCPair_n\tiltimes\NCPair_n$ by composing the functor from Remark \ref{R.FNCBipart} with the functor from Lemma~\ref{L.FPair'}. It remains to show that it is an isomorphism up to negligible morphisms.

First, it is clearly surjective since its image contains $\Pabba\times\Paabb=F(\Gconnecter)$, $\Paabb\times\Pabba=F(\Gwhite\Gconnecter)$, $\pairpart\times\pairpart=F(\pairpart)$, which are according to Lemma~\ref{L.twistgen} generators of $\NCPair_n\tiltimes\NCPair_n$.

Secondly, it must also be injective up to negligible morphisms. This is because there is a fibre functor $F_H\colon\NCBipartEven_{n^2}\to\Mat$, which factors through $\langle\TLconnecter,\LTconnecter,\TLpair\rangle$ (see Example~\ref{E.transpose}). Since the category $\NCBipartEven_{n^2}$ is pure, the fibre functor must be injective up to negligible morphisms.
\end{proof}

In order to get rid of the ``up to negligible morphisms'' and actually prove Theorem~\ref{T.Catalan} for arbitrary complex $\delta\neq 0$, we need to dig more deeply into the combinatorics and prove that the dimensions coincide by explicitly describing how this functor acts on the diagrams.

So, take any diagram $t\in\NCBipartEven(k,l)$. We are going to construct a pair of non-crossing pairings $p,q\in\NCPair(k,l)$. Recall that $k+l$ must be even, otherwise there is no diagram. The construction will work for arbitrary $k$ and $l$ but note that, without loss of generality, we can assume that both $k$ and $l$ are even, which might make the considerations a bit easier. (We could also just assume that $k=0$, but that actually does not make anything simpler.) 
Recall that $t$ is essentially a planar graph. More precisely, if we imagine that all inputs and outputs are connected to an additional common vertex, we get a planar graph. Now, we can colour every other face of this graph by yellow colour. Such a colouring makes sense: If we connect all the input and output strings to a single new vertex, we get a planar graph. It is known that the dual graph of a planar graph where every vertex has even degree is bipartite. For uniqueness, assume that the leftmost face is uncoloured. An example: 
$$
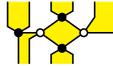
\Large\Graph{
\draw[fill=yellow] (5,1.5) -- (5,1) -- (4.5,0.5) -- (6,0.5) -- (6,1.5);
\draw[fill=yellow] (3.5,0) -- (2.5,0.5) -- (3.5,1) -- (4.5,0.5)node[fill=white]{} -- cycle;
\draw[fill=yellow] (3,-.5) -- (3,-.2) -- (3.5,0)node{} -- (4,-.2) -- (4,-.5);
\draw[fill=yellow] (3,1.5) -- (3,1.2) -- (3.5,1)node{} -- (4,1.2) -- (4,1.5);
\draw[fill=yellow] (1.5,-0.5) -- (1.5,0.5) -- (2.5,0.5)node[fill=white]{} -- (2,0.3) -- (2,-0.5);
\draw[fill=yellow] (1,1.5) -- (1,0.7) -- (1.5,0.5)node{} -- (2,0.7) -- (2,1.5);
}
$$

Now, we are going to construct the non-crossing pairing $p$ as follows: Follow the black lines, but whenever you hit a black vertex, connect your line with a neighbouring one going through the white region. Whenever you hit a white vertex, do the same, but go through the yellow region. To be more precise, for every black and white vertex, choose its small neighbourhood (e.g. shape of a circle) containing no other vertex. Your line should exatly follow the black one, but whenever you hit the neighbourhood of some vertex, follow its border through the white or yellow region until you hit a black line again. The pairing $q$ is constructed the same way, but switching black and white. In the example below, we denote $p$ in red and $q$ in blue for clarity. We also do not follow the black lines exactly for better readability.
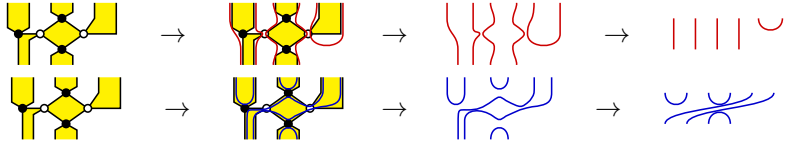
\begin{align*}
\Large\Graph{
\draw[fill=yellow] (5,1.5) -- (5,1) -- (4.5,0.5) -- (6,0.5) -- (6,1.5);
\draw[fill=yellow] (3.5,0) -- (2.5,0.5) -- (3.5,1) -- (4.5,0.5)node[fill=white]{} -- cycle;
\draw[fill=yellow] (3,-.5) -- (3,-.2) -- (3.5,0)node{} -- (4,-.2) -- (4,-.5);
\draw[fill=yellow] (3,1.5) -- (3,1.2) -- (3.5,1)node{} -- (4,1.2) -- (4,1.5);
\draw[fill=yellow] (1.5,-0.5) -- (1.5,0.5) -- (2.5,0.5)node[fill=white]{} -- (2,0.3) -- (2,-0.5);
\draw[fill=yellow] (1,1.5) -- (1,0.7) -- (1.5,0.5)node{} -- (2,0.7) -- (2,1.5);
}
\quad\rightarrow\quad
\Large\Graph{
\draw[fill=yellow] (5,1.5) -- (5,1) -- (4.5,0.5) -- (6,0.5) -- (6,1.5);
\draw[fill=yellow] (3.5,0) -- (2.5,0.5) -- (3.5,1) -- (4.5,0.5)node[fill=white]{} -- cycle;
\draw[fill=yellow] (3,-.5) -- (3,-.2) -- (3.5,0)node{} -- (4,-.2) -- (4,-.5);
\draw[fill=yellow] (3,1.5) -- (3,1.2) -- (3.5,1)node{} -- (4,1.2) -- (4,1.5);
\draw[fill=yellow] (1.5,-0.5) -- (1.5,0.5) -- (2.5,0.5)node[fill=white]{} -- (2,0.3) -- (2,-0.5);
\draw[fill=yellow] (1,1.5) -- (1,0.7) -- (1.5,0.5)node{} -- (2,0.7) -- (2,1.5);
\draw [red!80!black] (0.9,1.5) -- (0.9,1) .. controls (0.9,0.5) and (1.4,0.5) .. (1.4,0) -- (1.4,-.5);
\draw [red!80!black] (2.1,1.5) -- (2.1,0.7) .. controls (2.1,0.5) and (2.4,0.6) .. (2.4,0.5) .. controls (2.4,0.4) and (2.1,0.4) .. (2.1,0) -- (2.1,-.5);
\draw [red!80!black] (2.9,1.5) .. controls (2.9,1.4) .. (3,1.2) .. controls (3.1,1) and (3.1,0.85) .. (3,0.8) .. controls (2.4,0.5) .. (3,0.2) .. controls (3.1,0.15) and (3.1,0) .. (3,-.2) .. controls (2.9,-.4) .. (2.9,-.5);
\draw [red!80!black] (4.1,1.5) .. controls (4.1,1.4) .. (4,1.2) .. controls (3.9,1) and (3.9,0.85) .. (4,0.8) .. controls (4.6,0.5) .. (4,0.2) .. controls (3.9,0.15) and (3.9,0) .. (4,-.2) .. controls (4.1,-.4) .. (4.1,-.5);
\draw [red!80!black] (4.9,1.5) .. controls (4.9,0.8) and (4.6,1) .. (4.6,0.5) .. controls (4.6,0) and (6.1,0) .. (6.1,0.5) -- (6.1,1.5);
}
&\quad\rightarrow\quad
\Large\Graph{
\draw [red!80!black] (0.9,1.5) -- (0.9,1) .. controls (0.9,0.5) and (1.4,0.5) .. (1.4,0) -- (1.4,-.5);
\draw [red!80!black] (2.1,1.5) -- (2.1,0.7) .. controls (2.1,0.5) and (2.4,0.6) .. (2.4,0.5) .. controls (2.4,0.4) and (2.1,0.4) .. (2.1,0) -- (2.1,-.5);
\draw [red!80!black] (2.9,1.5) .. controls (2.9,1.4) .. (3,1.2) .. controls (3.1,1) and (3.1,0.85) .. (3,0.8) .. controls (2.4,0.5) .. (3,0.2) .. controls (3.1,0.15) and (3.1,0) .. (3,-.2) .. controls (2.9,-.4) .. (2.9,-.5);
\draw [red!80!black] (4.1,1.5) .. controls (4.1,1.4) .. (4,1.2) .. controls (3.9,1) and (3.9,0.85) .. (4,0.8) .. controls (4.6,0.5) .. (4,0.2) .. controls (3.9,0.15) and (3.9,0) .. (4,-.2) .. controls (4.1,-.4) .. (4.1,-.5);
\draw [red!80!black] (4.9,1.5) .. controls (4.9,0.8) and (4.6,1) .. (4.6,0.5) .. controls (4.6,0) and (6.1,0) .. (6.1,0.5) -- (6.1,1.5);
}
\quad\rightarrow\quad
\Large\Graph{
\draw[red!80!black] (1,0) -- (1,1);
\draw[red!80!black] (2,0) -- (2,1);
\draw[red!80!black] (3,0) -- (3,1);
\draw[red!80!black] (4,0) -- (4,1);
\draw[red!80!black] (4.9,1) .. controls (4.9,0.5) and (6,0.5) .. (6,1);
}\\
\Large\Graph{
\draw[fill=yellow] (5,1.5) -- (5,1) -- (4.5,0.5) -- (6,0.5) -- (6,1.5);
\draw[fill=yellow] (3.5,0) -- (2.5,0.5) -- (3.5,1) -- (4.5,0.5)node[fill=white]{} -- cycle;
\draw[fill=yellow] (3,-.5) -- (3,-.2) -- (3.5,0)node{} -- (4,-.2) -- (4,-.5);
\draw[fill=yellow] (3,1.5) -- (3,1.2) -- (3.5,1)node{} -- (4,1.2) -- (4,1.5);
\draw[fill=yellow] (1.5,-0.5) -- (1.5,0.5) -- (2.5,0.5)node[fill=white]{} -- (2,0.3) -- (2,-0.5);
\draw[fill=yellow] (1,1.5) -- (1,0.7) -- (1.5,0.5)node{} -- (2,0.7) -- (2,1.5);
}
\quad\rightarrow\quad
\Large\Graph{
\draw[fill=yellow] (5,1.5) -- (5,1) -- (4.5,0.5) -- (6,0.5) -- (6,1.5);
\draw[fill=yellow] (3.5,0) -- (2.5,0.5) -- (3.5,1) -- (4.5,0.5)node[fill=white]{} -- cycle;
\draw[fill=yellow] (3,-.5) -- (3,-.2) -- (3.5,0)node{} -- (4,-.2) -- (4,-.5);
\draw[fill=yellow] (3,1.5) -- (3,1.2) -- (3.5,1)node{} -- (4,1.2) -- (4,1.5);
\draw[fill=yellow] (1.5,-0.5) -- (1.5,0.5) -- (2.5,0.5)node[fill=white]{} -- (2,0.3) -- (2,-0.5);
\draw[fill=yellow] (1,1.5) -- (1,0.7) -- (1.5,0.5)node{} -- (2,0.7) -- (2,1.5);
\draw [blue!80!black] (1.1,1.5) -- (1.1,1) .. controls (1.1,0.5) and (1.9,0.5) .. (1.9,1) -- (1.9,1.5);
\draw [blue!80!black] (1.6,-.5) -- (1.6,0) .. controls (1.6,0.5) and (1.6,0.45) .. (2,0.45) .. controls (2.5,0.45) and (2.7,0.5) .. (3.2,0.75) .. controls (3.5,0.9) .. (3.8,0.75) .. controls (4.5,0.5) .. (4.9,0.75) .. controls (5.1,0.95) .. (5.1,1.2) -- (5.1,1.5);
\draw [blue!80!black] (1.9,-.5) -- (1.9,0.1) .. controls (1.9,0.5) and (2.7,0.5) .. (3.2,0.25) .. controls (3.5,0.1) .. (3.8,0.25) .. controls (4.3,0.4) and (4.3,0.45) .. (5,0.55) .. controls (5.7,0.6) and (5.9,0.7) .. (5.9,1) -- (5.9,1.5);
\draw [blue!80!black] (3.1,1.5) .. controls (3.1,1) and (3.9,1) .. (3.9,1.5);
\draw [blue!80!black] (3.1,-.5) .. controls (3.1,0) and (3.9,0) .. (3.9,-.5);
}
&\quad\rightarrow\quad
\Large\Graph{
\draw [blue!80!black] (1.1,1.5) -- (1.1,1) .. controls (1.1,0.5) and (1.9,0.5) .. (1.9,1) -- (1.9,1.5);
\draw [blue!80!black] (1.6,-.5) -- (1.6,0) .. controls (1.6,0.5) and (1.6,0.45) .. (2,0.45) .. controls (2.5,0.45) and (2.7,0.5) .. (3.2,0.75) .. controls (3.5,0.9) .. (3.8,0.75) .. controls (4.5,0.5) .. (4.9,0.75) .. controls (5.1,0.95) .. (5.1,1.2) -- (5.1,1.5);
\draw [blue!80!black] (1.9,-.5) -- (1.9,0.1) .. controls (1.9,0.5) and (2.7,0.5) .. (3.2,0.25) .. controls (3.5,0.1) .. (3.8,0.25) .. controls (4.3,0.4) and (4.3,0.45) .. (5,0.55) .. controls (5.7,0.6) and (5.9,0.7) .. (5.9,1) -- (5.9,1.5);
\draw [blue!80!black] (3.1,1.5) .. controls (3.1,1) and (3.9,1) .. (3.9,1.5);
\draw [blue!80!black] (3.1,-.5) .. controls (3.1,0) and (3.9,0) .. (3.9,-.5);
}
\quad\rightarrow\quad
\Large\Graph{
\draw[blue!80!black] (1,1) .. controls (1,0.5) and (2,0.5) .. (2,1);
\draw[blue!80!black] (3,0) .. controls (3,0.5) and (4,0.5) .. (4,0);
\draw[blue!80!black] (3,1) .. controls (3,0.5) and (4,0.5) .. (4,1);
\draw[blue!80!black] (1,0) .. controls (1,0.5) and (5,0.5) .. (5,1);
\draw[blue!80!black] (2,0) .. controls (2,0.5) and (6,0.5) .. (6,1);
}
\end{align*}

Let us denote this correspondence by $\alpha$. Notice that this assignment exactly corresponds to the functor from Proposition~\ref{P.Catalan} since what we do essentially is: look at the picture and replace every 
$\Graph{
\draw[fill=yellow] (1,0) -- (1,0.3) -- (1.5,0.5) -- (2,0.3) -- (2,0);
\draw[fill=yellow] (1,1) -- (1,0.7) -- (1.5,0.5) node {} -- (2,0.7) -- (2,1);
}$
by 
$
\Graph{
\draw[red!80!black] (1,0) -- (1,1);
\draw[red!80!black] (2,0) -- (2,1);
}
\times
\Graph{
\draw[blue!80!black] (1,0) .. controls (1,0.5) and (2,0.5) .. (2,0);
\draw[blue!80!black] (1,1) .. controls (1,0.5) and (2,0.5) .. (2,1);
}$
and every
$\Graph{
\draw[fill=yellow] (1,0) -- (1,0.3) -- (1.5,0.5) -- (2,0.3) -- (2,0);
\draw[fill=yellow] (1,1) -- (1,0.7) -- (1.5,0.5) node[fill=white] {} -- (2,0.7) -- (2,1);
}$
by
$\Graph{
\draw[red!80!black] (1,0) .. controls (1,0.5) and (2,0.5) .. (2,0);
\draw[red!80!black] (1,1) .. controls (1,0.5) and (2,0.5) .. (2,1);
}\times
\Graph{
\draw[blue!80!black] (1,0) -- (1,1);
\draw[blue!80!black] (2,0) -- (2,1);
}$.
We are not going to prove this formally here as this is not necessary to prove the theorem.

It remains to show that $\alpha$ is injective. We will do this by explicitly constructing its inverse. So, take a pair of non-crossing pairings $p,q\in\NCPair(k,l)$. We are going to construct the corresponding diagram $t\in\NCBipartEven(k,l)$ as follows.

Draw the two diagrams one on top of the other. Do it in such a way that no two wires would cross more than once. (It might happen that two points are paired together in both $p$ and $q$. In this case, the red string and the blue string are identical, which is fine.) This must always be possible. Indeed, instead of drawing the inputs on one line and outputs on some other line above, we could draw all the inputs and outputs on a single circle. Then the strings denoting the non-crossing pairings can be drawn using just line segments. Line segments can never cross more than once.
$$
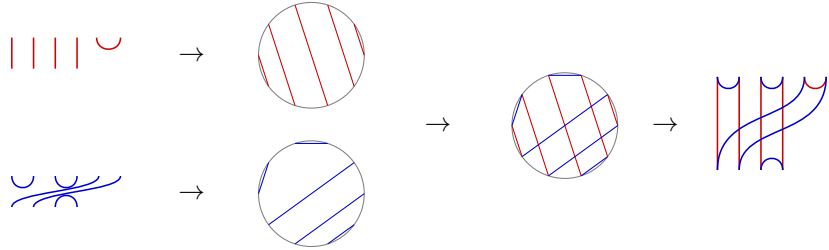
\begin{matrix}
\displaystyle\Large\Graph{
\draw[red!80!black] (1,0) -- (1,1);
\draw[red!80!black] (2,0) -- (2,1);
\draw[red!80!black] (3,0) -- (3,1);
\draw[red!80!black] (4,0) -- (4,1);
\draw[red!80!black] (4.9,1) .. controls (4.9,0.5) and (6,0.5) .. (6,1);
}
&\quad\rightarrow\quad&
\begin{tikzpicture}[baseline={(0,-.5ex)}]
  \draw[gray] (0,0) circle (2em);
  \foreach \i/\letter in {1/A,2/B,3/C,4/D,5/E,6/F,7/G,8/H,9/I,10/J}
  {
    \coordinate (\letter) at ({360/10 * (11-\i) + 180}:2em);
  }
  \draw[red!80!black] (A) -- (J);
  \draw[red!80!black] (B) -- (I);
  \draw[red!80!black] (C) -- (H);
  \draw[red!80!black] (D) -- (G);
  \draw[red!80!black] (E) -- (F);
\end{tikzpicture}\\\\
\displaystyle\Large\Graph{
\draw[blue!80!black] (1,1) .. controls (1,0.5) and (2,0.5) .. (2,1);
\draw[blue!80!black] (3,0) .. controls (3,0.5) and (4,0.5) .. (4,0);
\draw[blue!80!black] (3,1) .. controls (3,0.5) and (4,0.5) .. (4,1);
\draw[blue!80!black] (1,0) .. controls (1,0.5) and (5,0.5) .. (5,1);
\draw[blue!80!black] (2,0) .. controls (2,0.5) and (6,0.5) .. (6,1);
}
&\quad\rightarrow\quad&
\begin{tikzpicture}[baseline={(0,-.5ex)}]
  \draw[gray] (0,0) circle (2em);
  \foreach \i/\letter in {1/A,2/B,3/C,4/D,5/E,6/F,7/G,8/H,9/I,10/J}
  {
    \coordinate (\letter) at ({360/10 * (11-\i) + 180}:2em);
  }
  \draw[blue!80!black] (A) -- (B);
  \draw[blue!80!black] (C) -- (D);
  \draw[blue!80!black] (E) -- (J);
  \draw[blue!80!black] (F) -- (I);
  \draw[blue!80!black] (G) -- (H);
\end{tikzpicture}
\end{matrix}
\qquad\rightarrow\qquad
\begin{tikzpicture}[baseline={(0,-.5ex)}]
  \draw[gray] (0,0) circle (2em);
  \foreach \i/\letter in {1/A,2/B,3/C,4/D,5/E,6/F,7/G,8/H,9/I,10/J}
  {
    \coordinate (\letter) at ({360/10 * (11-\i) + 180}:2em);
  }
  \draw[red!80!black] (A) -- (J);
  \draw[red!80!black] (B) -- (I);
  \draw[red!80!black] (C) -- (H);
  \draw[red!80!black] (D) -- (G);
  \draw[red!80!black] (E) -- (F);
  \draw[blue!80!black] (A) -- (B);
  \draw[blue!80!black] (C) -- (D);
  \draw[blue!80!black] (E) -- (J);
  \draw[blue!80!black] (F) -- (I);
  \draw[blue!80!black] (G) -- (H);
\end{tikzpicture}
\quad\rightarrow\quad
\Large\Graph{
\draw[red!80!black] (1,-1) -- (1,2);
\draw[red!80!black] (2,-1) -- (2,2);
\draw[red!80!black] (3,-1) -- (3,2);
\draw[red!80!black] (4,-1) -- (4,2);
\draw[red!80!black] (5,2) .. controls (5,1.5) and (6,1.5) .. (6,2);
\draw[blue!80!black] (1,2) .. controls (1,1.5) and (2,1.5) .. (2,2);
\draw[blue!80!black] (3,-1) .. controls (3,-.5) and (4,-.5) .. (4,-1);
\draw[blue!80!black] (3,2) .. controls (3,1.5) and (4,1.5) .. (4,2);
\draw[blue!80!black] (1,-1) .. controls (1,1) and (5,0.5) .. (5,2);
\draw[blue!80!black] (2,-1) .. controls (2,0.5) and (6,0) .. (6,2);
}
$$

Now for each of the diagrams separately, we do a similar thing as we did before. We colour every other face of $p$ by red and every other face of $q$ by blue. There may be some regions, where the two colours overlap. These will be purple.

So, the whole diagram in total is now divided into white, red, blue and purple regions. Put a black vertex into every red region and a white vertex into every blue region. Finally, connect all vertices to the corners of their regions (the crossing points of red and blue lines). In the case when some blue and red line overlap (and hence do not cross with anything), you should simply connect its endpoints by a black line. Actually, even in the situation when the red and blue do cross, it will be convenient (for the purpose of Lemma~\ref{L.CatInj}) to allow the crossing point to be actually a curve segment (the red and blue meet, then they go together for a while, then they separate again). In that case, the black line should follow this curve segment.
$$
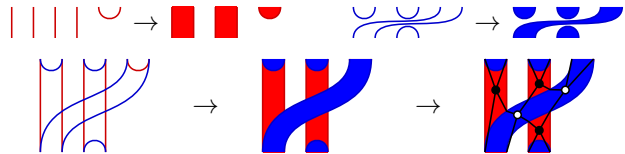
\Large\Graph{
\draw[red!80!black] (1,0) -- (1,1);
\draw[red!80!black] (2,0) -- (2,1);
\draw[red!80!black] (3,0) -- (3,1);
\draw[red!80!black] (4,0) -- (4,1);
\draw[red!80!black] (5,1) .. controls (5,0.5) and (6,0.5) .. (6,1);
}
\to
\Graph{
\fill[red,opacity=.5] (1,0)--(1,1)--(2,1)--(2,0);
\fill[red,opacity=.5] (3,0)--(3,1)--(4,1)--(4,0);
\draw[red!80!black] (1,0) -- (1,1);
\draw[red!80!black] (2,0) -- (2,1);
\draw[red!80!black] (3,0) -- (3,1);
\draw[red!80!black] (4,0) -- (4,1);
\draw[red!80!black,fill=red,fill opacity=.5] (5,1) .. controls (5,0.5) and (6,0.5) .. (6,1);
}
\qquad
\Graph{
\draw[blue!80!black] (1,1) .. controls (1,0.5) and (2,0.5) .. (2,1);
\draw[blue!80!black] (3,0) .. controls (3,0.5) and (4,0.5) .. (4,0);
\draw[blue!80!black] (3,1) .. controls (3,0.5) and (4,0.5) .. (4,1);
\draw[blue!80!black] (1,0) .. controls (1,1) and (5,0.1) .. (5,1);
\draw[blue!80!black] (2,0) .. controls (2,0.9) and (6,0) .. (6,1);
}
\to
\Graph{
\fill[blue,opacity=.5] (1,0) .. controls (1,1) and (5,0.1) .. (5,1) -- (6,1) .. controls (6,0) and (2,0.9) .. (2,0);
\draw[blue!80!black,fill=blue,fill opacity=.5] (1,1) .. controls (1,0.5) and (2,0.5) .. (2,1);
\draw[blue!80!black,fill=blue,fill opacity=.5] (3,0) .. controls (3,0.5) and (4,0.5) .. (4,0);
\draw[blue!80!black,fill=blue,fill opacity=.5] (3,1) .. controls (3,0.5) and (4,0.5) .. (4,1);
\draw[blue!80!black] (1,0) .. controls (1,1) and (5,0.1) .. (5,1);
\draw[blue!80!black] (2,0) .. controls (2,0.9) and (6,0) .. (6,1);
}
$$
$$
\Large\Graph{
\draw[red!80!black] (1,-1) -- (1,2);
\draw[red!80!black] (2,-1) -- (2,2);
\draw[red!80!black] (3,-1) -- (3,2);
\draw[red!80!black] (4,-1) -- (4,2);
\draw[red!80!black] (5,2) .. controls (5,1.5) and (6,1.5) .. (6,2);
\draw[blue!80!black] (1,2) .. controls (1,1.5) and (2,1.5) .. (2,2);
\draw[blue!80!black] (3,-1) .. controls (3,-.5) and (4,-.5) .. (4,-1);
\draw[blue!80!black] (3,2) .. controls (3,1.5) and (4,1.5) .. (4,2);
\draw[blue!80!black] (1,-1) .. controls (1,1) and (5,0.5) .. (5,2);
\draw[blue!80!black] (2,-1) .. controls (2,0.5) and (6,0) .. (6,2);
}
\quad\to\quad
\Graph{
\fill[red,opacity=.5] (1,-1)--(1,2)--(2,2)--(2,-1);
\fill[red,opacity=.5] (3,-1)--(3,2)--(4,2)--(4,-1);
\draw[red!80!black] (1,-1) -- (1,2);
\draw[red!80!black] (2,-1) -- (2,2);
\draw[red!80!black] (3,-1) -- (3,2);
\draw[red!80!black] (4,-1) -- (4,2);
\draw[red!80!black,fill=red,fill opacity=.5, name path=r5] (5,2) .. controls (5,1.5) and (6,1.5) .. (6,2);
\fill[blue,opacity=.5] (1,-1) .. controls (1,1) and (5,0.5) .. (5,2) -- (6,2) .. controls (6,0) and (2,0.5) .. (2,-1);
\draw[blue!80!black,fill=blue,fill opacity=.5] (1,2) .. controls (1,1.5) and (2,1.5) .. (2,2);
\draw[blue!80!black,fill=blue,fill opacity=.5] (3,-1) .. controls (3,-.5) and (4,-.5) .. (4,-1);
\draw[blue!80!black,fill=blue,fill opacity=.5] (3,2) .. controls (3,1.5) and (4,1.5) .. (4,2);
\draw[blue!80!black] (1,-1) .. controls (1,1) and (5,0.5) .. (5,2);
\draw[blue!80!black] (2,-1) .. controls (2,0.5) and (6,0) .. (6,2);
}
\quad\to\quad
\Graph{
\fill[red,opacity=.5] (1,-1)--(1,2)--(2,2)--(2,-1);
\fill[red,opacity=.5] (3,-1)--(3,2)--(4,2)--(4,-1);
\draw[red!80!black, name path=r1] (1,-1) -- (1,2);
\draw[red!80!black, name path=r2] (2,-1) -- (2,2);
\draw[red!80!black, name path=r3] (3,-1) -- (3,2);
\draw[red!80!black, name path=r4] (4,-1) -- (4,2);
\draw[red!80!black,fill=red,fill opacity=.5, name path=r5] (5,2) .. controls (5,1.5) and (6,1.5) .. (6,2);
\fill[blue,opacity=.5] (1,-1) .. controls (1,1) and (5,0.5) .. (5,2) -- (6,2) .. controls (6,0) and (2,0.5) .. (2,-1);
\draw[blue!80!black,fill=blue,fill opacity=.5, name path=b1] (1,2) .. controls (1,1.5) and (2,1.5) .. (2,2);
\draw[blue!80!black,fill=blue,fill opacity=.5, name path=b2] (3,-1) .. controls (3,-.5) and (4,-.5) .. (4,-1);
\draw[blue!80!black,fill=blue,fill opacity=.5, name path=b3] (3,2) .. controls (3,1.5) and (4,1.5) .. (4,2);
\draw[blue!80!black, name path=b4] (1,-1) .. controls (1,1) and (5,0.5) .. (5,2);
\draw[blue!80!black, name path=b5] (2,-1) .. controls (2,0.5) and (6,0) .. (6,2);
\draw (1,2) -- (1.5,1) node{} -- (1,-1);
\draw[name intersections={of=r2 and b4, by=A}] (2,2) -- (1.5,1) -- (A) -- (2.5,0.2) -- (2,-1);
\draw[name intersections={of=r3 and b4, by=B}, name intersections={of=r3 and b5, by=C}] (3,2) -- (3.5,1.2) node{} -- (B) -- (2.5,0.2)node[fill=white]{} -- (C) -- (3.5,-.3) node{} -- (3,-1);
\draw[name intersections={of=r4 and b4, by=D}, name intersections={of=r4 and b5, by=E}] (4,2) -- (3.5,1.2) -- (D) -- (4.7,1) -- (E) -- (3.5,-.3) -- (4,-1);
\draw (5,2) -- (4.7,1) node[fill=white]{} -- (6,2);
}
$$

First of all, note that thanks to the requirement that the red and blue line cannot cross more than once (and red cannot cross with red, blue cannot cross with blue), the resulting diagram does not depend on the exact way of how do we draw the two pairings. The reason is that the fact whether a given red line crosses a given blue one and in what order the crossings on a given line follow each other can be characterized combinatorially and does not depend on the particular drawing.

Now, the resulting bilabelled graph is clearly planar. The vertex colouring is surely proper as every black line goes through an intersection of a red and blue line, so it connects a red region with a blue one, so a black vertex with a white one. The degree of each vertex is clearly larger than two. Let us show that the degree is also even. All the crossings of red and blue lines in the picture are bordering with all four kinds of regions -- red, blue, white, and purple. In particular, this means that the edges surrounding every blue or red region always border with white or purple region in an alternating way. Hence, there must be an even number of the bordering edges, hence an even number of corners, hence an even number of black strings leaving the vertex.

Let us denote this correspondence $\beta\colon\NCPair\times\NCPair\to\NCBipartEven$.

\begin{lem}\label{L.CatInj}
It holds that $\beta\circ\alpha=\id$.
\end{lem}
\begin{proof}
Recall that in the definition of $\alpha$, we said that the blue and red lines of the newly constructed $p$ and $q$ exactly follow the black lines except for a small neighbourhood of the black and white vertices. Our claim is that the purple regions in the diagram after applying $\beta$ are contained in the yellow regions defined by $\alpha$. Moreover, they coincide except for the neighbourhoods of the black and white vertices. Well, by definition the leftmost region is always white. Now since the blue and red lines go always together with the black ones, this means that if you stand in the white region and cross the black line, you must enter the region painted both red and blue, i.e.\ purple. And vice versa. Consequently, the purple regions satisfy the defining property of the yellow regions.

Now what happens in the neighbourhood of the black and white vertices. Consider e.g.\ the black ones. Since the red line is supposed to go through the white region, the black vertex is behind the red line from the perspective of the white region, so must be painted red. On the other hand, the blue line goes through the yellow region, so the black vertex is behind the blue line from the perspective of the yellow, i.e. the purple region. Hence, it is \emph{not} painted blue. So, the black vertices are indeed in the red regions. Similarly, white vertices are in the blue regions. This is exactly, where the mapping $\beta$ would draw them.

Moreover, by definition of $\alpha$, the neighbourhoods of the black and white vertices do not contain any red and blue lines in their interior, so they are painted by a single colour. Consequently, the number of red and blue regions indeed equals the number of black and white vertices ($\beta\circ\alpha$ does not create any extra vertices).

Finally, in the small neighbourhood of every vertex, all the black lines that occur must clearly enter the vertex. Well this is again exactly what $\beta$ does. Outside the neighbourhoods, $\alpha$ maps black lines to red and purple on top of each other and $\beta$ does the reverse, so, in particular, $\beta\circ\alpha$ preserves the adjacency of vertices. This also applies for the situation when two points are paired in the original diagram, so there is no black or white vertex on the corresponding string.

This proves that the mapping $\beta$ indeed exactly recovers the original graph on which we first applied $\alpha$ which is what we wanted to show.
\end{proof}

\begin{lem}\label{L.Catalan1}
It holds that $\dim\NCBipartEven_{\delta^2}(k,l)\le(\dim\NCPair(k,l))^2$.
\end{lem}
\begin{proof}
As a direct consequence of Lemma~\ref{L.CatInj} we have that $\alpha$ is injective and hence we must have such an inequality for the dimensions.
\end{proof}

\begin{rem}
The reader might try to think about a similar proof for the equality $\alpha\circ\beta=\id$. Alternatively, it already follows from the fact that we essentially already know that $\alpha$ is surjective. See below.
\end{rem}

\begin{proof}[Proof of Theorem~\ref{T.Catalan}]
We can reuse most of the proof of Proposition~\ref{P.Catalan}. We construct the functor $F\colon\NCBipartEven_{\delta^2}\to\NCPair_\delta\tiltimes\NCPair_\delta$ the same way and prove its surjectivity the same way. The only thing that remains to prove is the injectivity of the functor. But since we already proved that it is surjective, the injectivity follows from Lemma~\ref{L.Catalan1}, where we showed that $\dim\NCBipartEven_{\delta^2}(k,l)\le\dim(\NCPair_\delta(k,l))^2=\dim((\NCPair_\delta\tiltimes\NCPair_\delta)(k,l))$.
\end{proof}

Finally, let us mention a couple of consequences this theorem has.

\begin{cor}\label{C.Catalan}
It holds that
$$\dim\NCBipartEven_{\delta^2}(k,l)=\begin{cases}C_{(k+l)/2}^2&\text{if $k+l$ is even}\\0&\text{if $k+l$ is odd.}\end{cases}$$
\end{cor}
\begin{proof}
Follows from the fact that $\dim\NCPair_\delta(k,l)=C_{(k+l)/2}$ if $k+l$ is even and 0 otherwise.
\end{proof}

\begin{cor}\label{C.NoNegligible}
The category $\NCBipartEven_{\delta^2}$ has negligible morphisms if and only if $\delta^2=4\cos^2(j\pi/l)$, $j=1,\dots,l-1$, $l>1$. In particular, any fibre functor on $\NCBipartEven_N$, $N\ge 4$ is injective.
\end{cor}
\begin{proof}
It is known \cite{GW03} (see also \cite[Fig.~1]{FM21}) that $\NCPair_\delta$ has negligible morphisms if and only if $\delta=2\cos(j\pi/l)$. We need to show that $\NCPair_\delta\tiltimes\NCPair_\delta$ has negligible morphisms under the same condition. The rest follows from the monoidal isomorphism.

Consider any $k\in2\N_0$. We need to show that the bilinear form $(p_1,p_2)=p_1\Ctrans p_2=R_k^\dag(p_1\otimes p_2)$, where $R_k=
\Graph{
	\draw (1,1) -- (1,0) -- (6,0) -- (6,1);
	\draw (1.5,1) -- (1.5,0.2) -- (5.5,0.2) -- (5.5,1);
	\draw (3,1) -- (3,0.5) -- (4,0.5) -- (4,1);
	\node[draw=none,fill=none] at (2.25,.8) {$\scriptstyle\dots$};
	\node[draw=none,fill=none] at (4.75,.8) {$\scriptstyle\dots$};
}$,
is non-degenerate in $\NCPair_\delta(0,k)$ if and only if analogous bilinear form on $(\NCPair_\delta\tiltimes\NCPair_\delta)(0,k)$ is non-degenerate. The latter bilinear form is given by
$$(p_1\times q_1,p_2\times q_2)=(p_1\times q_1)\Ctrans(p_2\times q_2)=(p_1\Ctrans p_2)\cdot(q_1\Ctrans q_2)$$

As we can see, if we choose a basis in $\NCPair_\delta(0,k)$ and denote the matrix associated to the bilinear form by $G$, then the matrix associated to $(\NCPair_\delta\tiltimes\NCPair_\delta)(0,k)$ is just $G\otimes G$, which is regular if and only if $G$ is regular.
\end{proof}

\section{Quantum symmetries}
\label{sec.qg}

\subsection{Quantum groups}
\label{secc.qgdef}

Quantum groups form a generalization of groups in non-commutative geometry in a similar way as quantum spaces generalize ordinary spaces. We will deal with the so-called \emph{compact quantum groups} as defined by Woronowicz \cite{Wor87}. To keep things simple, we are going to define only the orthogonal matrix version of quantum groups.

An~\emph{orthogonal compact matrix quantum group} is a pair $G=(\Alg,u)$, where $\Alg$ is a~$*$-algebra and $u=(u^i_j)\in M_N(\Alg)$ is a matrix with values in $\Alg$ such that
\begin{enumerate}
\item the entries $u^i_j$, $i,j=1,\dots, N$ generate $\Alg$,
\item the matrix $u$ is unitary and we have $\bar u:=(u^{i\,*}_j)_{i,j}=F^{-1}uF$ for some invertible matrix $F\in M_N(\C)$,
\item the map $\Delta\colon \Alg\to \Alg\otimes \Alg$ defined as $\Delta(u^i_j):=\sum_{k=1}^N u^i_k\otimes u^k_j$ extends to a~$*$-homomorphism.
\end{enumerate}

%For a~matrix group $G\subseteq M_N(\C)$, we define $u^i_j\colon G\to\C$ to be the coordinate functions $u^i_j(g):=g^i_j$. Then we define the \emph{coordinate algebra} $A:=\Olg(G)$ to be the algebra generated by $u^i_j$. The pair $(A,u)$ then forms a compact matrix quantum group. The so-called \emph{comultiplication} $\Delta\colon \Olg(G)\to \Olg(G)\otimes \Olg(G)$ dualizes matrix multiplication on $G$: $\Delta(f)(g,h)=f(gh)$ for $f\in \Olg(G)$ and $g,h\in G$.

The algebra~$\Alg$ should be seen as the algebra of non-commutative functions defined on some non-commutative compact underlying space. For this reason, we often denote $\Alg=\Olg(G)$ even if $\Alg$ is not commutative. The matrix $u$ is called the \emph{fundamental representation} of~$G$.

The condition $\bar u=F^{-1}uF$ essentially means that $u$ is self-conjugated. Indeed, assuming\footnote{This assumption is only necessary if we require that the cup $\uppairpart$ and cap $\pairpart$ are adjoints of each other. In general, taking any invertible matrix $F$, we can define $R^{ij}=F_i^j$ and $[R\Ctrans]_{ij}=[F^{-1}]_i^j$. This is a pair of duality morphisms satisfying the snake equation such that $u^*=F\bar uF^{-1}$ for arbitrary matrix $u$.} that $F^{-1}=\bar F$ and considering the duality morphism $R\in\C^N\otimes\C^N$ given by $R^{ij}=F_i^j$, then the conjugate $u^*$ (defined in Section~\ref{secc.qhad}; not the adjoint $u^\dag$) is exactly given by $u^*=F\bar uF^{-1}$.

%Actually, $\Olg(G)$ also has the structure of a~Hopf $*$-algebra. That is, there is a \emph{counit} $\epsilon\colon\Olg(G)\to\C$ mapping $u^i_j\mapsto\delta_{ij}$ (classical meaning is evaluating a function at identity, i.e.\ $\epsilon(f)=f(e)$) and an \emph{antipode} $S\colon\Olg(G)\to\Olg(G)$ mapping $u^i_j\mapsto [u^{-1}]^i_j$ (classical meaning is evaluating a function at the group inverse, i.e.\ $[S(f)](g)=f(g^{-1})$). In addition, we can also define the C*-algebra $C(G)$ as the universal C*-completion of $A$, which can be interpreted as the algebra of continuous functions of $G$.

In the next section, we are going to describe two things. First, how quantum groups can be used to describe symmetries. Secondly, that we actually do not have to deal with these $*$-algebras. Instead, quantum groups are equivalently described using their representation categories, which can be conveniently represented as diagrammatic categories (a.k.a. nice pictures -- a much better thing to deal with).

\subsection{Tannaka--Krein reconstruction and quantum automorphisms}

For a compact matrix quantum group $G=(\Olg(G),u)$, we say that $v\in M_n(\Olg(G))$ is a~representation of $G$ if $\Delta(v^i_j)=\sum_{k}v^i_k\otimes v^k_j$. The representation $v$ is called \emph{unitary} if it is unitary as a~matrix, i.e.\ $\sum_k v^i_kv^{j*}_k=\sum_k v^{k*}_iv^k_j=\delta_{ij}$. For instance, the fundamental representation $u$ is a unitary representation of $G$.

For two representations $v\in M_n(\Olg(G))$, $w\in M_m(\Olg(G))$ of $G$ we define the space of \emph{intertwiners}
$$\Mor(v,w)=\{T\colon \C^n\to\C^m\mid Tv=wT\}.$$
Since we are working with orthogonal compact matrix quantum groups only, it is actually enough to restrict our attention only to the tensor powers of $u$ since the entries of those representations already linearly span the whole $\Olg(G)$. So, we define a concrete category
$$\RCat_G(k,l):=\Mor(u^{\otimes k},u^{\otimes l})=\{T\colon (\C^N)^{\otimes k}\to(\C^N)^{\otimes l}\mid Tu^{\otimes k}=u^{\otimes l}T\}.$$

Conversely, we can reconstruct any compact matrix quantum group from its representation category \cite{Wor88,Mal18}:

\begin{thm}[Woronowicz--Tannaka--Krein]
\label{T.TK}
Let $\RCat\subset\Mat_N$ be a concrete category. Then there exists a unique orthogonal compact matrix quantum group $G$ such that $\RCat=\RCat_G$.
\end{thm}

We can write down the associated quantum group very concretely. The relations satisfied in the algebra $\Olg(G)$ will be exactly the intertwining relations:
$$\Olg(G)=\staralg(u^i_j,\;i,j=1,\dots,N\mid uu^\dag=1=u^\dag u,\; Tu^{\otimes k}=u^{\otimes l}T\;\forall T\in\RCat(k,l)).$$
If $S$ is a generating set of $\RCat_G$, we can actually use only the relations corresponding to the generators:
$$\Olg(G)=\staralg(u^i_j,\;i,j=1,\dots,N\mid uu^\dag=1=u^\dag u,\; Tu^{\otimes k}=u^{\otimes l}T\;\forall T\in S(k,l)).$$

As we already mentioned, quantum groups are used to describe (quantum) symmetries of some (quantum) structures. For instance, a quantum group $G$ is said to \emph{act} on a finite (quantum) space $X$ if there is a unital $*$-homomorphism $\alpha\colon C(X)\to C(X)\otimes\Olg(G)$ satisfying certain properties. The quantum automorphism group of $X$ is then defined to be the universal quantum group $G$ acting on $X$ (i.e.\ the largest possible acting faithfully).

Now the main point is that these certain properties can always be conveniently described in the categorical language. We want that the action to preserve the structure. That is, the action should commute with the structure maps. But commuting with the structure maps is nothing but some set of intertwiner relations. So, we can make the following vague definition:

Let $X$ be some (quantum) structure defined by structure maps $T_1,\dots,T_n$ with $T_i\colon l^2(X)^{\otimes k_i}\to l^2(X)^{\otimes l_i}$. Then the \emph{quantum automorphism group} of $X$ is the quantum group $\Aut^+X$ corresponding to the category $\RCat_{\Aut^+ X}=\langle T_1,\dots,T_n\rangle$.

For instance, a finite quantum space $X$ is given by a finite C*-algebra $C(X)$ (with multiplication $m\colon C(X)\otimes C(X)\to C(X)$) equipped with a certain functional $\eta^\dag\colon C(X)\to\C$. Hence, its quantum automorphism group $\Aut^+X$ is defined through $\RCat_{\Aut^+ X}=\langle m,\eta^\dag\rangle$. The convenient point about this definition is that this category is the image of $\NCPart_{\delta^2}$ under certain fibre functor. So, we can identify $m$ with $\Gfork$ and $\eta^\dag$ with $\Gsing$. See e.g.\ \cite{Ban99} or \cite[Section~4.3]{GQGraph} for details on this example. We can reformulate this particular example into a more explicit formal definition:
\begin{defn}
Let $X$ be a quantum space with a fixed orthonormal basis $(x_i)_{i=1}^N$ of $l^2(X)$. Denote $m_{ij}^k$ and $\eta_i$ the tensor entries of $m$ and $\eta^\dag$ in this basis, so $x_ix_j=\sum_k m_{ij}^k x_k$ and $\eta^\dag(x_i)=\eta_i$. We define the quantum automorphism group $\Aut^+X$ by
$$\Olg(\Aut^+X)=\staralg\left(u^i_j,i,j=1,\dots,N\,\middle|\,\begin{matrix}uu^\dag=1=u^\dag u,\;\sum_i\eta_iu^i_j=\eta_j\\\sum_{a,b}m^k_{ab}u^a_iu^b_j=\sum_c u^k_cm^c_{ij}\end{matrix}\right).$$
\end{defn}

If $X$ is the classical space of $N$ points, then $m_{ij}^k=\delta_{ijk}$ and $\eta_i=1$ and we obtain Wang's free quantum symmetric group $S_N^+$ \cite{Wan98}
$$\Olg(S_N^+)=\textstyle\staralg(u^i_j\mid (u^i_j)^2=u^i_j=u^{i*}_j,\;\sum_{j}u^i_j=1=\sum_iu^i_j).$$
Since the relations defining this quantum group are quite important, they deserved a name: A matrix $u$ (with possibly non-commutative entries) is called a \emph{magic unitary} or a \emph{quantum permutation matrix} if it is unitary and satisfies $(u^i_j)^2=u^i_j=u^{i*}_j,$ $\sum_{j}u^i_j=1=\sum_iu^i_j$.

Similarly, the quantum automorphism group of a finite graph was defined by Banica \cite{Ban05} and it amounts to adding the adjacency matrix to the category above. So, if $\GGr$ is a graph, then $\RCat_{\Aut^+\GGr}=\langle m,\eta^\dag,A\rangle$, where $m$ and $\eta^\dag$ correspond to the classical space and $A$ is the adjacency matrix. By the way, this category also has a nice diagrammatic realization \cite{MR20}. The explicit definition of the quantum automorphism group looks as follows:
$$\Olg(\Aut^+\GGr)=\textstyle\staralg(u^i_j\mid\text{$u$ is magic, }uA=Au).$$

There is also the concept of quantum graphs -- a quantum space $X$ equipped with a linear map $A\colon l^2(X)\to l^2(X)$ called the \emph{adjacency matrix} satisfying certain properties. Its quantum automorphism group is defined the same way -- as the quantum group $\Aut^+\GGr$ corresponding to the category $\RCat_{\Aut^+\GGr}=\langle m,\eta^\dag,A\rangle$.

\subsection{Hopf-bi-Galois objects and quantum isomorphisms}

In the classical world, the definitions of automorphisms and isomorphisms are closely related. The same should work for quantum isomorphisms. As we described above, a quantum automorphism of a finite quantum space $X$ is a unitary matrix $u\in M_n(\Alg)$ with entries in some $*$-algebra $\Alg$ such that $um=m(u\otimes u)$ and $\eta^\dag u=\eta^\dag$. So, given two finite quantum spaces $X$ and $X'$, a quantum isomorphism $X\to X'$ is simply a unitary matrix $u\in M_n(\Alg)$ such that $um=m'(u\otimes u)$ and $\eta^\dag u=\eta'^{\dag}$, where $m'$ and $\eta'^\dag$ denote the structure maps of $X'$.

This approach again has a categorical counterpart. Such a quantum isomorphism forms a so-called $\Aut^+ X$-$\Aut^+X'$-bi-Galois object. It was shown by Schauenburg \cite{Sch96} that the existence of a bi-Gaolis object is equivalent with having a monoidal equivalence between the corresponding representation categories. Let us reformulate the result in a more concrete way here:

\begin{thm}\label{T.Galois}
Suppose $\RCat=\langle T_1,\dots,T_n\rangle_N$ and $\RCat'=\langle T_1',\dots,T_n'\rangle_{N'}$ are two concrete categories such that $T_i\in\Mat_N(k_i,l_i)$ and $T_i'\in\Mat_{N'}(k_i,l_i)$. Then the map $T_i\mapsto T_i'$ extends to a monoidal $\dag$-isomorphism $\RCat\to\RCat'$ if and only if there exists a $*$-algebra $\Alg$ generated by the entries of a unitary matrix $u\in\Mat(N,N')\otimes \Alg$ satisfying $T_iu^{\otimes k_i}=u^{\otimes l_i}T_i'$.
\end{thm}
\begin{proof}
The easy direction is from right to left. Having such a matrix $u$, we can define a functor $F\colon\RCat\to\RCat'$ by $F(T)=u^{\dag\,\otimes l}Tu^{\otimes k}$ for any $T\in\RCat(k,l)$. It is straightforward to check that $F$ is indeed a monoidal unitary functor and that $F(T_i)=F(T_i')$.

For the opposite direction, see \cite[Theorem~2.3.11]{NT13}.
\end{proof}

\begin{rem}
Note the similarity of the above theorem with Tannaka--Krein duality (Theorem~\ref{T.TK}). In a sense this theorem can be seen as a generalization of Tannaka--Krein for \emph{quantum groupoids}. See e.g.\ \cite{BicHG} for more details on this viewpoint.
\end{rem}

Hence, we can say that two algebraic structures $X$ and $X'$ with structure maps $T_1,\dots,T_n$ and $T_1',\dots,T_n'$ are \emph{quantum isomorphic} if and only if the assignment $T_i\mapsto T_i'$ extends to a monoidal equivalence $\RCat_{\Aut^+X}\to\RCat_{\Aut^+X'}$. In particular, two quantum spaces $X$ and $X'$ are quantum isomorphic if and only if there is a monoidal isomorphism $\langle m,\eta^\dag\rangle\to\langle m',\eta'^\dag\rangle$ mapping $m\mapsto m'$, $\eta^\dag\mapsto\eta'^\dag$.

Actually, it holds that two quantum spaces are quantum isomorphic if and only if they have the same \emph{categorical dimension} $\delta^2=\eta^\dag\eta$. (So, for instance, the classical space $X_{n^2}=\{1,\dots,n^2\}$ is quantum isomorphic to the quantum space of $n\times n$ matrices $M_n$. We have seen the corresponding monoidal equivalence in Remark~\ref{R.Mn}.) 

\begin{prop}
Two finite quantum spaces $X$ and $X'$ are quantum isomorphic if and only if $\delta=\delta'$.
\end{prop}
\begin{proof}
By the proof of Proposition~\ref{P.FibreNCPart1} the category $\langle m,\eta^\dag\rangle$ is the image of $\NCPart_{\delta^2}$ under an appropriate fibre functor. Since the category $\NCPart_{\delta^2}$ is pure, by Proposition~\ref{P.pure} all fibre functors must have the same kernel, so their images must be monoidally isomorphic. Conversely, no monoidal isomorphism can map $\eta\mapsto\eta'$ if $\delta^2=\eta^\dag\eta\neq\eta'^\dag\eta'=\delta'^2$.
\end{proof}

We will use the same argumentation to prove analogous result for Hadamard matrices in Theorem~\ref{T.hadamard}.

As another example, suppose we have two classical graphs on $N$ vertices given by adjacency matrices $A$ and $A'$. We define them to be \emph{quantum isomorphic} by the following two equivalent conditions:
\begin{enumerate}
\item There is a magic unitary $u\in M_N(\Alg)$ such that $uA=A'u$.
\item There is a monoidal equivalence of categories $\langle m,\eta^\dag,A\rangle\to\langle m,\eta^\dag,A'\rangle$ mapping $m\mapsto m$, $\eta^\dag\mapsto\eta^\dag$, $A\mapsto A'$, where $m$ and $\eta^\dag$ correspond to the classical space of $N$ points.
\end{enumerate}
See also \cite{BCE+20} for more details on quantum isomorphisms of (quantum) graphs and \cite{MR20} for a remarkable combinatorial characterization of quantum isomorphisms of graphs.

\section{Quantum groups corresponding to $\NCBipartEven_N$}
\label{sec.qgH}

\subsection{Proving non-classicality}

In this section, we show that the quantum groups corresponding to $\NCBipartEven_N$ are not groups, but proper quantum groups.

The following lemma actually follows from \cite[Example~4.9]{GD4}, but we prove it in a more straightforward way using the results of Section~\ref{sec.cat}.
\begin{lem}
We have that $\NCBipartEven_N(2,2)=\spanlin\{\Paabb,\Pabba,\Gconnecter,\Gwhite\Gconnecter\}$.
\end{lem}
\begin{proof}
The inclusion $\supset$ is obvious, so it is enough to show that the dimension is correct $\dim\NCBipartEven_N(2,2)=4$. But this follows directly from Corollary~\ref{C.Catalan}.
\end{proof}

\begin{prop}
Consider $N>2$. Then for any fibre functor $F\colon\Bipart_N\to\Mat$ we have that $F(\crosspart)\not\in F(\NCBipartEven_N)$.
\end{prop}
\begin{proof}
We need to show that $F(\crosspart)$ is not a linear combination of the elements $F(\Paabb)$, $F(\Pabba)$, $F(\Gconnecter)$, $F(\Gwhite\Gconnecter)$, which span $F(\NCBipartEven_N(2,2))$ according to the lemma above. Hence, we must prove that these five elements are linearly independent. This is true if and only if the Gram matrix of these elements with respect to the Hilbert--Schmidt inner product $\langle f,g\rangle=\Tr(f^\dag g)$ is regular. The Hilbert--Schmidt inner product can be expressed solely inside the category without explicit knowledge of the fibre functor $F$ as
$$\det
{
\def\entry#1{
	\scriptscriptstyle\vrule height 0.9em depth 0.65em width 0pt
	\Graph{#1
		\draw (2,1.5) -- (2,1.7) -- (2.5,1.7) -- (2.5,-0.7) -- (2,-0.7) -- (2,-0.5);
		\draw (1,1.5) -- (1,1.9) -- (3,1.9) -- (3,-.9) -- (1,-.9) -- (1,-.5);}}
\def\row#1{
\entry{\draw(1,-.5) -- (1,-.2) -- (2,-.2) -- (2,-.5);\draw(1,0.5) -- (1,0.2) -- (2,0.2) -- (2,0.5);                                            #1}&
\entry{\draw(1,-.5) -- (1,0.5);\draw(2,-.5) -- (2,0.5);                                                                                        #1}&
\entry{\draw (1,-.5) -- (1,-.2) -- (1.5,0) -- (2,-.2) -- (2,-.5);\draw (1,0.5) -- (1,0.2) -- (1.5,0) node {} -- (2,0.2) -- (2,0.5);            #1}&
\entry{\draw (1,-.5) -- (1,-.2) -- (1.5,0) -- (2,-.2) -- (2,-.5);\draw (1,0.5) -- (1,0.2) -- (1.5,0) node[fill=white] {} -- (2,0.2) -- (2,0.5);#1}&
\entry{\draw (1,-.5) -- (2,0.5); \draw (2,-.5) -- (1,0.5);                                                                                     #1}\\
}
\begin{pmatrix}
\row{\draw(1,0.5) -- (1,0.8) -- (2,0.8) -- (2,0.5);\draw(1,1.5) -- (1,1.2) -- (2,1.2) -- (2,1.5);}
\row{\draw(1,0.5) -- (1,1.5);\draw(2,0.5) -- (2,1.5);}
\row{\draw (1,0.5) -- (1,0.8) -- (1.5,1) -- (2,0.8) -- (2,0.5);\draw (1,1.5) -- (1,1.2) -- (1.5,1) node {} -- (2,1.2) -- (2,1.5);}
\row{\draw (1,0.5) -- (1,0.8) -- (1.5,1) -- (2,0.8) -- (2,0.5);\draw (1,1.5) -- (1,1.2) -- (1.5,1) node[fill=white] {} -- (2,1.2) -- (2,1.5);}
\row{\draw (1,0.5) -- (2,1.5); \draw (2,0.5) -- (1,1.5);}
\end{pmatrix}
}
=\det
\begin{pmatrix}
N^2&N&N&N&N\\
N&N^2&N&N&N\\
N&N&N&1&N\\
N&N&1&N&N\\
N&N&N&N&N^2
\end{pmatrix}
=N^3(N-1)^4(N-2)
$$

So, the Gram matrix is regular if and only if $N\neq 0,1,2$.
\end{proof}

\begin{rem}
The assumption $N>2$ seems too restrictive as we could have just excluded the possibilities $N=0,1,2$. But actually there is no fibre functor unless $N>2$ or $N=1,2$ (the category is undefined for $N=0$). Indeed, as we mentioned in Remark~\ref{R.FibreExistence}, the fibre functor $F$ can exist only if the above Gram matrix is positive semidefinite. The Gram matrix is singluar if and only if $N=1,2$ -- in this case the fibre functor clearly exists, so the matrix is actually positive semidefinite. Computing the leading minors, we get that the Gram matrix is positive definite if and only if $N>2$.
\end{rem}

\begin{cor}\label{C.NCBQ}
Consider $N>2$ and a fibre functor $F\colon\Bipart_N\to\Mat$ such that $F(\crosspart)$ is the flip map. Then the quantum group corresponding to $F(\NCBipartEven_N)$ is not a group, but a proper quantum group.
\end{cor}

\subsection{Case $N=4$}
In this section, we are going to show that the quantum group corresponding to the category $F_H(\NCBipartEven_4)$ for some Hadamard matrix $H$ is the anticommutative $SO_4^{-1}$. First, recall what are the generators of this representation category.

We denote by $\hat\Pabcd$ a certain linear combination of non-crossing partitions (diagrams involving black spiders only) such that $[F_4(\hat\Pabcd)]^{ij}_{kl}$ equals to one if and only if all the indices are mutually distinct and zero otherwise. See \cite[Sections 1.6, 2.1]{GD4}. In addition, we denote $\tilde\crosspart:=2\Gconnecter-\crosspart$, so that
$$[F_4(\tilde\crosspart)]^{ij}_{kl}=\begin{cases}1&\text{if $i=j=k=l$}\\-1&\text{if $i=l\neq j=k$}\\0&\text{otherwise}\end{cases}$$
With this notation, we can say the following:

\begin{prop}[{\cite[Example~2.5]{GD4}}]\label{P.SO4}
The representation category of $SO_4^{-1}$ is given by $F_4(\langle\hat\Pabcd,\tilde\crosspart,\pairpart\rangle)$.
\end{prop}

Now, consider the matrices
$$H=
\begin{pmatrix}
1&-1&-1&-1\\
-1&1&-1&-1\\
-1&-1&1&-1\\
-1&-1&-1&1
\end{pmatrix},
\qquad
\F=
\begin{pmatrix}
1&1&1&1\\
1&-1&1&-1\\
1&1&-1&-1\\
1&-1&-1&1
\end{pmatrix}.
$$
Both are Hadamard of the size four. The first one can actually be expressed as $H=F_4(2\idpart-\Gdisconnecter)$ ($F_4$ being the standard interpretation of black spiders) while the second one is actually the Fourier transform on $\Z_2\times\Z_2$. Note in addition that both are actually self-adjoint, so $H=H^\dag=4H^{-1}$, $\F=\F^\dag=4\F^{-1}$.

\begin{lem}\label{L.SO4}
With the notation above, we have that
\begin{align*}
F_H(\Gconnecter)&=\frac{1}{4}\F^{\otimes 2} F_4(-\tilde\crosspart+\hat\Pabcd+\Paabb+\Pabba)\F^{-1\,\otimes 2}\\
F_H(\Gwhite\Gconnecter)&=\frac{1}{4}\F^{\otimes 2}F_4(-\tilde\crosspart-\hat\Pabcd+\Paabb+\Pabba)\F^{-1\,\otimes 2}
\end{align*}
\end{lem}
\begin{proof}
These equalities can be checked in a straightforward way as they only involve adding and multiplying some $16\times 16$ matrices. But, let us sketch also a more abstract argument.

First, recall that $F_H(\Gconnecter)=F_4(\Gconnecter)$. Denote $\Graph{\draw (1,0) -- (1,0.5) \Gmapq{} -- (1,1);}:=\idpart-\frac{1}{2}\Gdisconnecter$, which corresponds to the normalized Hadamard matrix $\frac{1}{2}H=F_4(\Graph{\draw (1,0) -- (1,0.5) \Gmapq{} -- (1,1);})$. Thus, $F_H({\Gwhite\Gconnecter})=F_4(
\Graph{
\draw (1,-.3) -- (1,0) \Gmapq{} -- (1,0.3) -- (1.5,0.5) -- (2,0.3) -- (2,0) \Gmapq{} -- (2,-.3);
\draw (1,1.3) -- (1,1) \Gmapq{} -- (1,0.7) -- (1.5,0.5) node {} -- (2,0.7) -- (2,1) \Gmapq{} -- (2,1.3);
})$ (cf. \cite[Section~3.2]{GD4}). Finally, note that
$$[\F^{-1\,\otimes l}F_4(\spider)\F^{\otimes k}]^{j_1,\dots,j_l}_{i_1,\dots,i_k}=4^{1-l}\delta_{i_1\oplus\cdots\oplus i_k,j_1\oplus\cdots\oplus j_l},$$
where $\oplus$ denotes the group operation on $\Z_2\times\Z_2$ (this actually works for any abelian group and its Fourier transform, see \cite[eq.~(3.2)]{GroAbSym}).

Now this allows us to check this relatively easily by hand: Adding and subtracting the equalities from the statement, we equivalently need to show that
\begin{align*}
\F^{-1\,\otimes 2}F_H(\Gconnecter-\Gwhite\Gconnecter)\F^{\otimes 2}&=\frac{1}{2}F_4(\hat\Pabcd),\\
\F^{-1\,\otimes 2}F_H(\Gconnecter+\Gwhite\Gconnecter)\F^{\otimes 2}&=\frac{1}{2}F_4(-\tilde\crosspart+\Paabb+\Pabba).
\end{align*}

So, for the first equation, we have
\begin{align*}
&[\F^{-1\,\otimes 2}F_H(\Gconnecter-{\Gwhite\Gconnecter})\F^{\otimes 2}]^{ij}_{kl}=[\F^{-1\,\otimes 2}F_4(\Gconnecter-
\Graph{
\draw (1,-.3) -- (1,0) \Gmapq{} -- (1,0.3) -- (1.5,0.5) -- (2,0.3) -- (2,0) \Gmapq{} -- (2,-.3);
\draw (1,1.3) -- (1,1) \Gmapq{} -- (1,0.7) -- (1.5,0.5) node {} -- (2,0.7) -- (2,1) \Gmapq{} -- (2,1.3);
})\F^{\otimes 2}]^{ij}_{kl}\\
&=[\F^{-1\,\otimes 2}F_4\left(\frac{1}{2}(\Pabbb+\Pabaa+\Paaba+\Paaab)-\frac{1}{4}(\Paabc+\Pabac+\dots)+\frac{1}{4}\Pabcd\right)\F^{\otimes 2}]^{ij}_{kl}\\
&=\frac{1}{2}(\delta_{j,e}\delta_{i,k\oplus l}+\delta_{i,e}\delta_{j,k\oplus l}+\dots)-(\delta_{i,j}\delta_{k,e}\delta_{l,e}+\dots)+4\delta_{i,e}\delta_{j,e}\delta_{k,e}\delta_{l,e},
\end{align*}
where $e$ denotes the (index corresponding to) the group identity in $\Z_2\times\Z_2$. It is quite easy to check that this equals to $1/2$ if all the indices are mutually distinct and zero otherwise. The second equality can be checked in a similar manner.
\end{proof}

\begin{thm}\label{T.SO4}
Consider arbitrary Hadamard matrix $H$ of size 4. Then the quantum group corresponding to the category $F_H(\NCBipartEven_4)$ is isomorphic to $SO_4^{-1}$.
\end{thm}
\begin{proof}
For the particular choice of $H$, we made above, it follows from Lemma~\ref{L.SO4} that $F_H(\NCBipartEven_4)=F_H(\langle\Gconnecter,{\Gwhite\Gconnecter},\pairpart\rangle)=F_4(\hat\Pabcd,\tilde\crosspart,\pairpart)$, which is by Proposition~\ref{P.SO4} exactly the representation category of $SO_4^{-1}$.

It is well known that all Hadamard matrices of size four are mutually equivalent (see Section~\ref{secc.hadamard} for a definition). This means that the images of the corresponding functors $F_H$ differ only by conjugation with a certain matrix. Hence, the corresponding quantum groups must be isomorphic.
\end{proof}

%NOT TRUE!!!
%\begin{cor}
%There is a fibre functor $F:\NCBipart_4\to\Mat_4$ such that the corresponding quantum group equals to $SO_4$.
%\end{cor}

%\begin{rem}
%More explicitly, the fibre functor maps
%$$...$$
%This is an example of a fibre functor $\NCBipart_N\to\Mat$ that does not factor through $\NCHad_N$. Indeed, if it did, then the corresponding quantum group would have to be $SO_4^{-1}$ according to \ref{P.SO4}.
%\end{rem}

\begin{rem}
It was recently shown in \cite{CE23} that the representation category of the classical $SO_4$ can be modelled by $\NCPair_2\times\NCPair_2$, which corresponds to the well known fact that the representation categories of $q$-deformed quantum groups (e.g.\ $SO_4$ and $SO_4^{-1}$) are equivalent, but not necessarily monoidally equivalent.
\end{rem}

\begin{rem}
In \cite{GD4}, we interpreted the corresponding quantum group as the free quantum version of Coxeter group $D_4$ as we did not know that it is actually isomorphic to $SO_4^{-1}$. Although there is no mistake in the article \cite{GD4}, we can see that this interpretation is somewhat misleading: $SO_4^{-1}$ is definitely not free -- it obeys some (anti)commutation relations. The problem lies in the fact that although we proved that the flip map is not contained in the category and, moreover, the category can be modelled by non-crossing diagrams, this does not necessarily imply that there are no commutation relations. In our case, the deformed commutation relations correspond to $\tilde\crosspart$, i.e.
$$-2\Gconnecter-2\Gwhite\Gconnecter+\Paabb+\Pabba.$$
\end{rem}

\section{Hadamard matrices and their symmetries}
\label{sec.main}

\subsection{Hyperoctahedral (quantum) group}

In the following text, we will denote by $H_N$ the hyperoctahedral group -- the symmetry group of the $N$-dimensional hypercube. Structurally, it can be written as a wreath product $H_N=\Z_2\wr S_N$. Therefore, it can be realized as a matrix group by $N\times N$ signed permutation matrices. That is, matrices with entries $\pm 1$ or zero such that in each row and column there is only one nonzero entry. As can be easily seen, this is equivalent to writing
$$H_N=\{P\in O_N\mid P^i_jP^i_k=0=P_i^jP_i^k\text{ for $j\neq k$}\}.$$

Let's say that a matrix $P$ (with possibly non-commutative entries) is \emph{cubic} if it is orthogonal and satisfies the relation $P^i_jP^i_k=0=P_i^jP_i^k$. It is then natural to define the \emph{free quantum hyperoctahedral group} $H_N^+$ by
$$\Olg(H_N^+)=\staralg(u^i_j\mid\text{$u$ is cubic}).$$

The relation $P^i_jP^i_k=0=P_i^jP_i^k$ can be equivalently written as $(P\otimes P)T=T(P\otimes P)$ with $T=F_N(\Gconnecter)$. Thus, $H_N$ is the quantum group with representation category given by $F_N(\langle\Gconnecter,\crosspart,\pairpart\rangle)$ while $H_N^+$ corresponds to the category $F_N(\langle\Gconnecter,\pairpart\rangle)$.

This quantum group was first introduced by Bichon in \cite{Bic04}, where he also defined the free quantum counterpart of the wreath product. By the way, the free hyperoctahedral quantum group is actually \emph{not} the quantum automorphism group of the $N$-dimensional hypercube graph, but it \emph{is} the quantum automorphism group of a graph given by $N$ segments \cite{BBC07}.

\subsection{Hadamard matrices}
\label{secc.hadamard}

Two Hadamard matrices $H$ and $H'$ are called \emph{equivalent} if one can be transformed to the other using the following operations: permuting rows, permuting columns, multiplying a row by $-1$, multiplying a column by $-1$. It is easy to see that this transformation can be encoded using two signed permutation matrices $P,Q\in H_N$ such that $H'=PHQ^{-1}$.

This motivates the definition of an \emph{automorphism group} $\Aut H$ of a given Hadamard matrix $H$. Note that if $H=PHQ^{-1}$, then the matrix $P$ is already determined by $Q$ as $P=HQH^{-1}$. This allows the following convenient description of $\Aut H$ as a matrix group:
$$\Aut H=\{Q\in H_N\mid HQH^{-1}\in H_N\}=H_N\cap H^{-1}H_NH.$$

Now it is already straightforward to quantize this definition.

\begin{defn}
Let $H$ be an $N\times N$ Hadamard matrix. We define the \emph{quantum automorphism group} of $H$ to be $\Aut^+H=H_N^+\cap H^{-1}H_N^+H$.
\end{defn}

That is, it is a compact matrix quantum group, where the associated Hopf $*$-algebra can be defined using the following relations
$$\Olg(\Aut^+ H)=\staralg(u_{ij}\mid\text{$u$ and $HuH^{-1}$ are cubic}).$$

Alternatively, one can describe the quantum group in terms of the associated category
$$\RCat_{\Aut^+H}=F_H(\langle\Gconnecter,
\Graph{
\draw (1,-0.2) -- (1,0) \Gmapsc{} -- (1,0.3) -- (1.5,0.5) -- (2,0.3) -- (2,0) \Gmapsc{} -- (2,-0.2);
\draw (1,1.2) -- (1,1) \Gmapscr{} -- (1,0.7) -- (1.5,0.5) node {} -- (2,0.7) -- (2,1) \Gmapscr{} -- (2,1.2);
},
\pairpart\rangle)
=F_H(\langle\Gconnecter,\Gwhite\Gconnecter,\pairpart\rangle).$$

\begin{defn}
Two Hadamard matrices $H$ and $H'$ of size $N$ and $N'$ respectively are \emph{quantum isomorphic} if there exists a $*$-algebra $\Alg$ generated by elements of a cubic matrix $q\in\Mat(N,N')\otimes \Alg$ such that $p=H'qH^{-1}$ is also a cubic matrix.
\end{defn}

\begin{thm}\label{T.hadamard}
Two Hadamard matrices are quantum isomorphic if and only if they have the same size.
\end{thm}
\begin{proof}
As follows from Theorem~\ref{T.Galois}, two Hadamard matrices are quantum isomorphic if and only if there exist a monoidal isomorphism $\RCat_{\Aut^+X}\to\RCat_{\Aut^+X'}$ mapping generators to generators.

If two Hadamard matrices are supposed to be quantum isomorphic, they have to have the same size as the size can be expressed in terms of the associated category as $\pairpart\cdot\uppairpart$.

On the other hand, suppose that two Hadamard matrices have the same size $N$. As we indicated above, the representation category $\RCat_G$ corresponding to $G=\Aut^+H$ is the image of $\NCBipartEven_N=\langle\Gconnecter,\Gwhite\Gconnecter,\pairpart\rangle$ under an appropriate fibre functor. Since the diagrammatic category $\NCBipartEven_N$ is pure, all the fibre functors have the same kernel according to Proposition~\ref{P.pure}. Hence, the images must be isomorphic.
\end{proof}

%For the definition of an equivalence of such matrices, we of course also allow multiplying rows or columns by any complex unit instead of just $\pm 1$. Consequently, we define its automorphism group as
%$$\Aut_\C H=\{P\in K_N\mid H^{-1}PH\in K_N\}=K_N\cap HK_NH^{-1},$$
%where $K_N=\T\wr S_N$ is the group of permutation matrices with entries in $\T$.
%
%Denoting $K_N^+:=\T\wr_* S_N^+$, we define the following.
%
%\begin{defn}
%Let $H$ be an $N\times N$ complex Hadamard matrix. We define the \emph{complex quantum automorphism group} of $H$ as $\Aut_\C^+H=K_N^+\cap HK_N^+H^{-1}$.
%\end{defn}
%
%Categorical description...

\subsection{Hadamard graphs}
\label{secc.graphs}

Let $H$ be an $N\times N$ Hadamard matrix. We define a graph $\HGr$ called a \emph{Hadamard graph} corresponding to each such matrix as follows \cite[Section~1.8]{BCN89}.

The set of vertices is given by $V=\{r_i^+,r_i^-,c_i^+,c_i^-\}_{i=1}^N$, i.e.\ we have $4N$ vertices. Here, the letter $r$ stands for \emph{row} and $c$ stands for \emph{column}. For every $i,j=1,\dots,N$, we have a pair of edges between $r_i^\pm$ and $c_j^\pm$ if $H_{ij}=+1$ and a pair of edges between $r_i^\pm$ and $c_j^\mp$ if $H_{ij}=-1$. There are no other edges. (Hence, the graph is bipartite with row vertices and column vertices forming the two parts.)

We can write down the adjacency matrix of $\HGr$ in a block-wise form as follows, where the blocks stand for $r^+, r^-,c^+,c^-$:
$$A=\begin{pmatrix}
0&0&H^+&H^-\\
0&0&H^-&H^+\\
H^{+\rm T}&H^{-\rm T}&0&0\\
H^{-\rm T}&H^{+\rm T}&0&0
\end{pmatrix}.
$$

Here, $H^+=\frac{1}{2}(J+H)$, where $J$ is the all-one-matrix. That is, the entries of $H^+$ are \emph{zeros} and \emph{ones} with \emph{one} exactly in the places, where $+1$ is in $H$. Similarly, $H^-=\frac{1}{2}(J-H)$ having \emph{one} exactly in those places, where $H$ has $-1$.

The idea is that the graph not only encodes the structure of the Hadamard matrix, but equivalence operations on Hadamard matrices correspond exactly to permutations of vertices preserving the two parts: multiplying $i$-th row (resp.\ column) by $-1$ corresponds to swapping $r_i^+$ with $r_i^-$ (resp.\ $c_i^+$ with $c_i^-$), swapping $i$-th row (resp.\ column) with $j$-th row (resp.\ column) corresponds to swapping $(r_i^+,r_i^-)$ with $(r_j^+,r_j^-)$ (resp.\ $(c_i^+,c_i^-)$ with $(c_j^+,c_j^-)$).

This means that if we really want to reconstruct a Hadamard matrix from a Hadamard graph, we need, in addition, to keep track of which vertices correspond to rows and which correspond to column. We may either colour them or put a loop to all row vertices (and have no loop at the column vertices). We will denote such a graph by $\HGr_0$ and call it the \emph{looped Hadamard graph}. Its adjacency matrix will be denoted by $A_0$.

\begin{prop}[{\cite{McK79}}]
There is a one-to-one correspondence between equivalence classes of Hadamard matrices and isomorphism classes of looped Hadamard graphs.
\end{prop}

Now, we would like to formulate a similar result for quantum automorphisms and isomorphisms.

\begin{lem}\label{L.HadGr}
Let $H$ be a Hadamard matrix and $\HGr$, $\HGr_0$ the corresponding Hadamard graphs. Then $\Aut^+ H$ acts on both $\HGr$ and $\HGr_0$. More precisely, consider the following matrix
$$u=\begin{pmatrix}
p^+&p^-&0&0\\
p^-&p^+&0&0\\
0&0&q^+&q^-\\
0&0&q^-&q^+
\end{pmatrix},\qquad
\begin{matrix}
[p^\pm]^i_j=\frac{1}{2}((p^i_j)^2\pm p^i_j),\\
[q^\pm]^i_j=\frac{1}{2}((q^i_j)^2\pm q^i_j),
\end{matrix}
$$
where $q$ is the fundamental representation of $\Aut^+H$, $p=HqH^{-1}$. Then $u$ is a magic unitary and satisfies $uA=Au$, $uA_0=A_0u$.
\end{lem}
\begin{proof}
First, it is straightforward to check that $u$ is magic. (See the proof of \cite[Theorem~6.2]{BBC07}.)

Secondly, we need to prove that $uA=Au$ and $uA_0=A_0u$. We will prove the statement with $\HGr_0$. The case for $\HGr$ is literally the same. So, we expand the left-hand and right-hand side first:
$$
A_0u=
\begin{pmatrix}
I&0&H^+&H^-\\
0&I&H^-&H^+\\
H^{+\rm T}&H^{-\rm T}&0&0\\
H^{-\rm T}&H^{+\rm T}&0&0
\end{pmatrix}
\begin{pmatrix}
p^+&p^-&0&0\\
p^-&p^+&0&0\\
0&0&q^+&q^-\\
0&0&q^-&q^+
\end{pmatrix}
=
\begin{pmatrix}
p^+&p^-&\heartsuit&\diamondsuit\\
p^-&p^+&\diamondsuit&\heartsuit\\
\clubsuit&\spadesuit&0&0\\
\spadesuit&\clubsuit&0&0
\end{pmatrix},
$$
\begin{align*}
\heartsuit&=H^+q^++H^-q^-,&\diamondsuit&=H^+q^-+H^-q^+,\\
\clubsuit&=H^{+\rm T}p^++H^{-\rm T}p^-,&\spadesuit&=H^{+\rm T}p^-+H^{-\rm T}p^+.
\end{align*}

Multiplying $uA_0$ looks the same with
\begin{align*}
\heartsuit&=p^+H^++p^-H^-,&\diamondsuit&=p^+H^-+p^-H^+,\\
\clubsuit&=q^+H^{+\rm T}+q^-H^{-\rm T},&\spadesuit&=q^+H^{-\rm T}+q^-H^{+\rm T}.
\end{align*}

So, we have four relations to prove. Let us look, for instance, on the heart relation.
\begin{align*}
H^+q^++H^-q^-&=\frac{1}{2}(J+H)q^++\frac{1}{2}(J-H)q^-=\frac{1}{2}J(q^++q^-)+\frac{1}{2}H(q^+-q^-)\\
&=\frac{1}{2}J+\frac{1}{2}Hq=\frac{1}{2}J+\frac{1}{2}pH=\dots=p^+H^++p^-H^-
\end{align*}
Here, we used the fact that $J(q^++q^-)=J=(p^++p^-)J$, which follows from orthogonality of $q$ and $p$.

All the other relations are proven in a similar way.
\end{proof}

\begin{prop}
Let $H$ be a Hadamard matrix and $\HGr_0$ the corresponding looped Hadamard graph. Then $\Aut^+\HGr_0=\Aut^+ H$.
\end{prop}
\begin{proof}
We already proved the inclusion $\supset$ in the previous lemma, so it remains to prove $\subset$. That is, we need to prove that $\Olg(\Aut^+\HGr_0)$ is a quotient of $\Olg(\Aut^+ H)$.

So, denote by $u$ the fundamental representation of $\Aut^+\HGr_0$. By \cite[Lemma 3.2.3]{Ful06}, quantum automorphisms preserve cycles of length $k$. In particular, it must preserve vertices with loops. That is, $u_{ij}=0$ if $i$-th vertex has a loop and $j$-th does not or the other way around. Consequently, $u$ has the form
$$u=
\begin{pmatrix}
a&b&0&0\\
c&d&0&0\\
0&0&\tilde a&\tilde b\\
0&0&\tilde c&\tilde d
\end{pmatrix},
$$
where $a$, $b$, $c$, $d$, $\tilde a$, $\tilde b$, $\tilde c$, $\tilde d$ are some $N\times N$ matrices.

Expanding the relation $uA=Au$, we get eight equations as follows
\begin{align*}
(a+b)J+(a-b)H&=J(\tilde a+\tilde c)+H(\tilde a-\tilde c)\\
(a+b)J-(a-b)H&=J(\tilde b+\tilde d)+H(\tilde b-\tilde d)\\
(c+d)J+(c-d)H&=J(\tilde a+\tilde c)-H(\tilde a-\tilde c)\\
&\dots
\end{align*}

From the definition of $\Aut^+\HGr$, we have that $u\eta=\eta$ and hence $uJ=J$ (as $J=\eta\eta^\dag$). Consequently, $(a+b)J=J$, $(c+d)J=J$ and so on, so we can cancel the first term on each side of each of the equation above. From what remains, it is easy to derive that $a=d$, $b=c$, $\tilde a=\tilde d$, $\tilde b=\tilde c$.% Hence, $u$ is of the form
%$$u=
%\begin{pmatrix}
%a&b&0&0\\
%b&a&0&0\\
%0&0&\tilde a&\tilde b\\
%0&0&\tilde b&\tilde a
%\end{pmatrix}.
%$$

It is straightforward to check (see the proof of \cite[Theorem~6.2]{BBC07}) that $p:=a-b$ and $q:=\tilde a-\tilde b$ are cubic matrices. If we substitute this to one of the equations above, we get $pH=Hq$, which is the last relation we needed to derive.
\end{proof}

We say that a graph $\GGr$ \emph{has quantum symmetries} if $\Aut^+\GGr$ is ``larger'' than $\Aut\GGr$. More precisely, if $\Olg(\Aut^+\GGr)$ is non-commutative, so $\Aut^+\GGr$ is a proper quantum group.

\begin{prop}
For $N\ge 4$, considering a Hadamard matrix of size $N$, both the corresponding Hadamard graphs $\HGr$ and $\HGr_0$ have quantum symmetries.
\end{prop}
\begin{proof}
By Corollary~\ref{C.NCBQ}, $\Aut^+ H$ is a proper quantum group. By Lemma~\ref{L.HadGr}, it acts on both $\HGr$ and $\HGr_0$.
\end{proof}

\begin{thm}\label{T.graphs}
Two Hadamard graphs or two looped Hadamard graphs are quantum isomorphic if and only if they have the same size
\end{thm}
\begin{proof}
It is known that quantum isomorphisms preserve the number of vertices of a graph, so we have the left-right implication.

For the right-left implication, take two Hadamard matrices of the same size. By Theorem~\ref{T.hadamard}, there is a quantum isomorphism mapping $H$ to $H'$, so there is a pair of cubic matrices $p$ and $q$ with non-commutative entries such that $H'q=pH$. Now we define $u$ as in Lemma~\ref{L.HadGr} and show that $A'u=uA$ (or $A_0'u=uA_0$). The proof of this can be copied from Lemma~\ref{L.HadGr} (just adding the primes to one side of the equations).
\end{proof}

\begin{rem}
As we mentioned in the introduction, this theorem was at the same time independently proved by \cite{CM22}. Although their proof is written in a ``different language'', the crux of the proof is similar to ours. Instead of studying quantum isomorphisms of Hadamard matrices, they focus just on the graphs. The category associated to any graph can be modelled by all planar bilabelled graphs \cite{MR20}. For every graph $\GGr$, there is a fibre functor $F_\GGr$ given essentially by sending every edge to the adjacency matrix of $\GGr$. (The article \cite{CM22} does not refer to monoidal categories at all; instead the images of the diagrams under $F_\GGr$ are called \emph{scaffolds}.) They need to show that for any two Hadamard graphs of a given size, the associated categories are monoidally equivalent. The diagrammatic category given by all planar bilabelled graphs is not pure (otherwise all graphs of a given size would be mutually quantum isomorphic); nevertheless, to prove that two fibre functors have the same kernel, it is enough to show that they map closed diagrams the same way. This is exactly what they do in \cite{CM22}. In fact, they do it in a more general way and show that any two \emph{exactly triply regular association schemes} are quantum isomorphic if and only if they have the same \emph{Delta-Wye parameters}.
\end{rem}

\subsection{Quantum Hadamard matrices and graphs}

All the results of this section can be reformulated in a straightforward way to the case of quantum Hadamard matrices and graphs.

Given a quantum Hadamard matrix $H$ defined over a finite quantum space $X$, we define its quantum automorphism group $\Aut^+ H$ through the category $\RCat_{\Aut^+ H}=F_H(\langle\Gconnecter,\Gwhite\Gconnecter,\pairpart\rangle)$. More concretely, it is given by the $*$-algebra
$$\Olg(\Aut^+H)=\staralg(u_{ij}\mid\text{$u$ and $HuH^{-1}$ are $X$-cubic}).$$
Here, $u$ being $X$-cubic means that it is unitary, it satisfies $(u\otimes u)F_X(\uppairpart)=F_X(\uppairpart)$ (is $X$-orthogonal) and $(u\otimes u)F_X(\Gconnecter)=F_X(\Gconnecter)(u\otimes u)$, where $F_X$ denotes the fibre functor mapping $\Gfork\mapsto m$, $\Gsing\mapsto\eta^\dag$.

Two quantum Hadamard matrices $H$ and $H'$ over quantum spaces $X$ and $X'$ are called \emph{quantum isomorphic} if there is a unitary matrix $q$ with non-commutative entries which is $X$-$X'$-cubic and $p:=H'qH^{-1}$ is also $X$-$X'$-cubic. That is, the following relations are satisfied for both $u=p,q$
$$(u\otimes u)F_X(\uppairpart)=F_{X'}(\uppairpart),\quad (u\otimes u)F_X(\Gconnecter)=F_{X'}(\Gconnecter)(u\otimes u).$$

By the same argumentation as in Theorem~\ref{T.hadamard}, we obtain that
\begin{thm}
Two quantum Hadamard matrices $H$ and $H'$ are quantum isomorphic if and only if the underlying finite quantum spaces $X$ and $X'$ are quantum isomorphic (i.e.\ iff $\delta=\delta'$).
\end{thm}

Now for every quantum Hadamard matrix $H$, we can construct the matrices $H^\pm:=\frac{1}{2}(J\pm H)$, where $J=\eta\eta^\dag$. Since $H=H^*$, we also have $H^{\pm *}=H^\pm$. In addition,
$$H^\pm\bullet H^\pm=\frac{1}{4}(J\bullet J\pm J\bullet H\pm H\bullet J+H\bullet H)=H^\pm,$$
where we used the fact that $J\bullet A=A=A\bullet J$ for any $A$.

This is the quantum analogue of the property that the matrices $H^\pm$ consist only of zeros and ones. By the way, this means that we can consider them to encode two directed quantum graphs, where one is the complement of the other. But this is of course \emph{not} the Hadamard graph in the sense of the definition in the previous subsection.

In order to construct a quantum Hadamard graph, we need to introduce a new quantum space $Y:=X\sqcup X\sqcup X\sqcup X$, that is,
$$C(Y)=C(X)\oplus C(X)\oplus C(X)\oplus C(X)=\C^4\otimes C(X).$$

Now, we can define two quantum adjacency matrices $A,A_0\colon l^2(Y)\to l^2(Y)$ by the same formula as classically:
$$A=\begin{pmatrix}
0&0&H^+&H^-\\
0&0&H^-&H^+\\
H^{+\dag}&H^{-\dag}&0&0\\
H^{-\dag}&H^{+\dag}&0&0
\end{pmatrix}
,\qquad
A_0=\begin{pmatrix}
I&0&H^+&H^-\\
0&I&H^-&H^+\\
H^{+\dag}&H^{-\dag}&0&0\\
H^{-\dag}&H^{+\dag}&0&0
\end{pmatrix}.
$$

It is straightforward to check that both $A$ and $A_0$ satisfy $A\bullet A=A$, $A=A^*$, $A=A^\dag$, so they can be considered as undirected quantum graphs $\HGr$ and $\HGr_0$ on $Y$. In addition, $A$ also satisfies $A\bullet I=0$, where $I$ is the identity matrix, so it encodes a quantum graph with no loops. See \cite{GQGraph} for a definition of a quantum graph.

Now all the statements about Hadamard graphs must hold also for quantum Hadamard graphs since all the relations that are to be checked have exactly the same form. In particular, we have the following.

\begin{prop}
Let $H$ be a quantum Hadamard matrix and $\HGr$, $\HGr_0$ the corresponding quantum Hadamard graphs. Then $\Aut^+ H$ acts on both $\HGr$ and $\HGr_0$. Moreover, $\Aut^+\HGr_0=\Aut^+ H$.
\end{prop}

In this case, the action goes via the $Y$-magic unitary
$$u=\begin{pmatrix}
p^+&p^-&0&0\\
p^-&p^+&0&0\\
0&0&q^+&q^-\\
0&0&q^-&q^+
\end{pmatrix},\qquad
\begin{matrix}
p^\pm=\frac{1}{2}(p\bullet p\pm p),\\
q^\pm=\frac{1}{2}(q\bullet q\pm q),\\
\end{matrix}
$$
where $q$ is the fundamental representation of $\Aut^+H$, $p=HqH^{-1}$.

\begin{thm}
Two quantum Hadamard graphs are quantum isomorphic if and only if the underlying finite quantum spaces $X$ and $X'$ are quantum isomorphic (i.e.\ iff $\delta=\delta'$).
\end{thm}

\section{Concluding remarks and open problems}
\label{sec.conclusion}

First of all, let us mention that this article answers most of the open questions we raised in \cite{GD4}. We
\begin{itemize}
\item found a lot of fibre functors for $\NCBipartEven_N$ and hence answered Question 5.4;
\item described the structure of $\NCBipartEven_{\delta^2}$ as a product $\NCPair_\delta\tiltimes\NCPair_\delta$ and hence answered Question 5.7 solving Conjectures 5.8, 5.9;
\item this also answers the semisimplicity questions 5.5 and 5.6 as we know the answers for $\NCPair_\delta$;
\item we interpreted the quantum groups as quantum symmetries of Hadamard graphs, which answers Question 5.12.
\end{itemize}

Nevertheless, our work also raises new questions. First, did we found all the fibre functors?

\begin{quest}
Is there a fibre functor $\NCBipartEven_N\to\Mat$ that does not factor through $\NCHad_N$?
\end{quest}

Secondly, we showed that for $N=4$, the corresponding quantum group is actually isomorphic to $SO_4^{-1}$. Can we say something more about the structure of the other quantum groups?

\begin{quest}
What is the structure of the quantum groups corresponding to $\NCBipartEven_N$?
\end{quest}

This is particularly interesting from the following viewpoint: Non-equivalent Hadamard matrices often have non-isomorphic automorphism groups. Therefore, their quantum automorphism groups must be non-isomorphic as well. But we showed that they are monoidally equivalent. So, this brings many examples of monoidal equivalences among quantum groups. Can we describe them in some systematical manner?

Regarding the quantum isomorphism of Hadamard matrices or Hadamard graphs, we proved that it exists, but we did not find it explicitly. In particular, it is not clear, whether the quantum isomorphism can be realized via a finite-dimensional algebra in the following sense:\footnote{This is sometimes called being \emph{quantum tensor isomorphic}. Actually, in the quantum information literature, the finite-dimensionality assumption is often quite important. Hence, this notion is often referred to as the \emph{quantum isomorphism} and our original definition is called the \emph{quantum commuting isomorphism}. Compare with \cite{AMR19}.}

\begin{quest}
Given some two Hadamard matrices $H$, $H'$ of size $N$, is there a finite-dimensional Hilbert space $\mathscr{H}$ and a set of linear operators $q_{ij}$, $i,j=1,\dots,N$ on $\mathscr{H}$ such that $q=(q_{ij})$ is cubic and $p:=H'qH^{-1}$ is cubic?
\end{quest}

Finally, we introduced quantum Hadamard graphs, but did not study them much. They certainly deserve more attention. Most importantly, we should look for more examples. It is not clear, for which quantum spaces a Hadamard matrix can exist. We only have them for
\begin{itemize}
\item $X$ being a classical space of size $4n$ (this is actually also open, whether there is one for every $n$ -- the famous Hadamard conjecture),
\item $X=M_n$ the quantum space of $n\times n$ matrices -- here we have the transposition example (Example~\ref{E.transpose}),
\item tensor product of the above constructions.
\end{itemize}

\begin{quest}
Find more examples of quantum Hadamard matrices. Is there one for $X$ not being of the form above? In particular, is there one for a finite quantum space with non-tracial state $\eta^\dag$?
\end{quest}

Classically, Hadamard graphs with $4N$ vertices (corresponding to a Hadamard matrix of size $N$) are exactly the distance-regular graphs with intersection array $(N,N-1,N/2,1;1,N/2,N-1,N)$ \cite[Section~1.8]{BCN89}.

\begin{quest}
Is there a similar abstract characterization of quantum Hadamard graphs?
\end{quest}

\bibliographystyle{halpha}
\bibliography{mybase}

\end{document}